\documentclass[a4paper, 12pt]{amsart}
\usepackage{amsmath, amsthm, amscd, amssymb, amsfonts, amsxtra, amssymb, latexsym}
\usepackage{enumerate}
\usepackage{verbatim}
\usepackage{float}
\usepackage{graphicx,epsfig}
\usepackage{tikz}

\usepackage{array}
\usepackage{hyperref}
\hypersetup{colorlinks = true,	allcolors  = blue}

\newcommand{\G}{\Gamma}
\newcommand{\Z}{\mathbb{Z}}

\newcommand{\N}{\mathbb{N}}
\newcommand{\ff}{\mathbb{F}}
\newcommand{\Tr}{\mathrm{Tr}}

\newcommand{\sk}{\smallskip}

\newtheorem{thm}{Theorem}[section]
\newtheorem{prop}[thm]{Proposition}
\newtheorem{lem}[thm]{Lemma}
\newtheorem{coro}[thm]{Corollary}

\theoremstyle{definition}
\newtheorem{rem}[thm]{Remark}
\newtheorem{exam}[thm]{Example}

\theoremstyle{remark}

\usepackage{color}

\begin{document} \sloppy
\numberwithin{equation}{section}
\title[Connected components and non-bipartiteness of GP-graphs]{Connected components and non-bipartiteness \\ of generalized Paley graphs}
\author[R.A.\@ Podest\'a, D.E.\@ Videla]{Ricardo A.\@ Podest\'a, Denis E.\@ Videla}
\dedicatory{\today}
\thanks{Partially supported by CONICET and SECyT-UNC}

\address{Ricardo A.\@ Podest\'a, FaMAF -- CIEM (CONICET), Universidad Nacional de C\'ordoba, \newline
	Av.\@ Medina Allende 2144, Ciudad Universitaria, (5000) C\'ordoba, Argentina. 
	\newline {\it E-mail: podesta@famaf.unc.edu.ar}}
\address{Denis E.\@ Videla, FaMAF -- CIEM (CONICET), Universidad Nacional de C\'ordoba, \newline
	Av.\@ Medina Allende 2144, Ciudad Universitaria,  (5000) C\'ordoba, Argentina. 
	\newline {\it E-mail: devidela@famaf.unc.edu.ar}}

\begin{abstract}
In this work we consider the class of Cayley graphs known as generalized Paley graphs (GP-graphs for short) given by $\G(k,q) = Cay(\ff_q, \{x^k : x\in \ff_q^* \})$, where $\ff_q$ is a finite field with $q$ elements,
both in the directed and undirected case. Hence $q=p^m$ with $p$ prime, $m\in \N$ and one can assume that $k\mid q-1$.
We first give the connected components of an arbitrary GP-graph. We show that these components are smaller GP-graphs all isomorphic to each other (generalizing Lim and Praeger's result from 2009 to the directed case). We then characterize those GP-graphs which are disjoint unions of odd cycles.
Finally, we show that $\G(k,q)$ is non-bipartite except for the graphs $\G(2^{m-1},2^m)$, $m \in \N$, which are isomorphic to $K_2 \sqcup \cdots \sqcup K_2$, the disjoint union of $2^{m-1}$ copies of $K_2$.
\end{abstract}

\maketitle

\section{Introduction} 
Paley graphs form a classic family of Cayley graphs. For $q$ a prime power, the associated Paley graph 
	$$\mathcal{P}_q = Cay(\ff_q,(\ff_q^*)^2)$$
is the graph with the finite field of $q$ elements as vertex set and where two vertices $x$ and $y$ form an edge if $y-x$ is a non-zero square in the field. 
We will consider a natural generalization of Paley graphs and study two basic structural properties: the connected components and bipartiteness for these graphs.

In this work, $\ff_q$ will denote a finite field with $q$ elements, where $q=p^m$ with $p$ prime and $m \in \N$.
A \textit{generalized Paley graph} (GP-graph for short) is the Cayley graph
\begin{equation} \label{eq: Gkq}
	\G(k,q) = Cay(\ff_{q},R_k) \qquad \text{with} \qquad R_k=\{ x^{k} : x \in \ff_{q}^*\} = (\ff_q^*)^k.
\end{equation}
By definition, two vertices $u,v$ form a  oriented edge (or arrow, or arc) $\vec{uv}$ from $u$ to $v$ if $v-u \in R_k$ and thus $\G(k,q)$ is directed.
If $R_k$ is symmetric, i.e.\@ $R_k=-R_k$ (equivalently if $-1\in R_k$), $\vec{uv}$ is an arrow if and only if $\vec{vu}$ is an arrow. 
Hence, in this case, it is customary to think of both arrows (the oriented 2-cycle) as a single undirected edge $uv$ and to consider $\G(k,q)$ as undirected. Formally, we make the identification 
\begin{equation} \label{eq: ident}
	\{\{u,v\}\}=\{(u,v), (v,u)\}. 
\end{equation}

It is usual to assume that $k\mid q-1$ (and we do that from now on), since 
\begin{equation} \label{eq: G=G'}
	\G(k,q)=\G(\gcd(k,q-1),q).	
\end{equation}
Hence, given $q$, there are $\# \mathrm{div}(q-1)$ different 
GP-graphs defined over $\ff_q$, where div$(m)$ is the set of positive divisors of $m$. 

The graph $\G(k,q)$ is $n$-regular with 
	$n=\frac{q-1}{k}$ (a directed graph is $n$-regular if in each vertex the in-degree and the out-degree equals $n$) and has 
no loops, since $0\not \in R_k$. 
Also, it is well-known that  \label{conditions}
\begin{equation} \label{eq: C1C2}
\begin{aligned}
& \hspace{0cm} (\text{C1}). \:\: \text{$\G(k,q)$ is connected if and only if $ord_n(p) = m$, and} \\[1mm]
& \hspace{0cm} (\text{C2}). \:\: \text{$\G(k,q)$ is undirected if and only if $q$ is even or else $q$ is odd and $k \mid \tfrac{q-1}2$}. 
\end{aligned}
\end{equation}
Hence, $\G(k,q)$ is directed if $q$ is odd and $k \nmid \tfrac{q-1}2$. 
Note that if $\G(k,q)$ is directed and $k$ is even then $\G(\frac k2,q)$ is undirected, since $\frac k2 \mid \frac{q-1}2$. 
 In this case we have
\begin{equation} \label{eq: G dirs}
	\G(\tfrac k2,q) = \overset{_{\rightarrow}}{\G}(k,q) \cup \overset{_{\leftarrow}}{\G}(k,q)
\end{equation} 
where $\overset{_{\rightarrow}}{\G}(k,q)$ denotes the oriented graph $\G(k,q)$ 
with the given orientation and $\overset{_{\leftarrow}}{\G}(k,q)$ is the graph $\overset{_{\rightarrow}}{\G}(k,q)$ with the 
opposite orientation (i.e., the arrows reversed). Also, notice that $\overset{_{\rightarrow}}{\G}(k,q)$, $\overset{_{\leftarrow}}{\G}(k,q)$ are $n$-regular and that $\G(\tfrac k2,q)$ is $2n$-regular.

Besides connectedness, another basic and important property of graphs to study is bipartiteness. By a classic result of Brualdi, Harary and Miller, these properties are closely related. The bipartite double of a (di)graph $X$, denoted $B(X)$, is a bipartite double cover of $X$ that can be realized as the product $X \times K_2$. Brualdi et al.\@ proved that $B(X)$ is connected if and only if $X$ is connected and non-bipartite (\cite{BHM}). 
For instance, the four cycle $C_4$ is bipartite and $B(C_4)=C_4 \times K_2=C_4 \sqcup C_4$ is non-connected.

If $H$ is a subgraph of $G$ we denote it by $H\subset G$.
It is clear that if $k\mid \ell$ then $\G(\ell,q) \subset \G(k,q)$ 
if both graphs are directed or both graphs are undirected 
(also if one is directed and the other not, by \eqref{eq: ident}).
Similarly, if $a \mid m$, viewing $\ff_{p^a}$ as a subfield of $\ff_{p^m}$, 
we have that $\G(k,p^a) \subset \G(k,p^m)$. 
Thus, in general, if $k\mid \ell$ and $a \mid m$, then we have 
\begin{equation} \label{eq: GP-subgraphs}	
	\G(\ell,p^a) \subset \G(k,p^m).
\end{equation}

Note that $\G(1,q)$ is the complete graph in $q$-vertices $K_q$. If $q\equiv 1 \pmod 4$, then $\G(2,q)$ is the classic undirected Paley graph $\mathcal{P}_q$, which is a self-complementary strongly regular graph. If $q\equiv 3 \pmod 4$, then $\G(2,q)$ is the directed Paley graph $\vec{\mathcal{P}}_q$, which is a tournament (a directed graph with exactly one edge between each two vertices, in one of the two possible directions), hence transitive, acyclic and with one Hamiltonian path. 
Note that, by \eqref{eq: G dirs}, $\vec{\mathcal{P}}_q$ is an oriented clique.
The graphs $\G(3,q)$ and $\G(4,q)$ were studied before and have some interest 
(see \cite{PV3}, \cite{PV8}). 
There are other important families of GP-graphs, such as the Hamming GP-graphs (see \cite{LP}, \cite{PV4}), Cartesian decomposable GP-graphs (see \cite{PP}, \cite{PV4})
and the semiprimitive GP-graphs (see \cite{PV3}, \cite{PV4b}; and \cite{BWX} for a generalization).

The \textit{spectrum} of a graph $\G$, denoted $Spec(\G)$, is the set of eigenvalues of its adjacency matrix $A$ counted with multiplicities.
If $\Gamma$ has different eigenvalues $\lambda_0, \ldots, \lambda_t$ with multiplicities $m_0,\ldots,m_t$, we write as usual 
\begin{equation} \label{spec}
	Spec(\Gamma) = \{[\lambda_0]^{m_0}, \ldots, [\lambda_t]^{m_t}\}.
\end{equation}
It is well-known that the connectedness of a graph can be read from the spectrum.
If $\G$ is $n$-regular, then $n$ is the greatest eigenvalue, with multiplicity equal to the number of connected components of $\G$. 
Thus, $\G$ is connected if and only if $n$ has multiplicity $1$.
Similarly, for $n$-regular digraphs, i.e.\@ those directed graphs such that any vertex has the same in-degree and out-degree equal to $n$, $\G$ is strongly connected if and only if $n$ has multiplicity $1$.
Recently, we have expressed the spectrum of a general GP-graph $\G(k,q)$ in terms of cyclotomic Gaussian periods (see Theorem 2.1 in \cite{PV3}).

\subsubsection*{Outline and results} 
In Section \ref{sec:2} we study the connected components of the GP-graphs. In Theorem \ref{thm: components} we show that when the graph $\G(k,p^m)$ is not (strongly) connected, it is the disjoint union of $p^{m-a}$ isomorphic copies of a single GP-graph $\G(k_a,p^a)$ where 	
	$$k_a=\tfrac{p^a-1}{n}, \qquad n=\tfrac{p^m-1}{k} \qquad \text{and} \qquad a=ord_n(p),$$ 
thus generalizing the result of Lim and Praeger for undirected graphs (\cite{LP}) to the general case. Namely, if $\sqcup$ stands for disjoint union of graphs, we have  
	$$\G(k,q) = \bigcup_{h \in \ff_q} \G[F_a+h]  \simeq \G(k_a,p^a) \sqcup \cdots \sqcup \G(k_a,p^a)$$   
with $p^{m-a}$ terms, where $\G[F_a+h]$ is the induced subgraph of $\G(k,q)$ by the coset $F_a+h$ of the subfield $F_a = \{x \in \ff_q : x^{p^a}=x\}$.
In Examples \ref{ex: cliques}--\ref{ex: cliques Kpt} we give GP-graphs whose connected components are cliques and in Example \ref{ex: qml} we give two connected families of GP-graphs (which are strongly regular).

In Section \ref{sec:3}, using Theorem \ref{thm: components}, 
we characterize those GP-graphs which are disjoint unions of $p$-cycles with $p$ an odd prime. 
In Proposition \ref{prop: cycle decompositions} we show that 
	$$\G(\tfrac{q-1}2,q) \simeq C_p \sqcup \cdots \sqcup C_p \qquad \text{and} \qquad 
\G(q-1,q) \simeq \vec{C}_p \sqcup \cdots \sqcup \vec{C}_p$$ 
($p^{m-1}$ times), where $C_p$ and $\vec{C}_p$ (the $p$-cycle and the oriented 
$p$-cycle respectively) are (isomorphic to) the connected components. Conversely, if a GP-graph is a disjoint union of $p$-cycles, then it must be one of the graphs $\G(k,q)$ with $k \in \{\frac{q-1}2,q\}$ and $q=p^m$ for some $m\in\N$. 
{In Examples \ref{exam: GP3s}--\ref{exam: GP7s} we study the connectedness (by giving the decomposition into connected components) of the GP-graphs $\G(k,p^m)$ for $p=3,5,7$ and small values of $m$.

Finally, in Section \ref{sec:5} we show that the graphs $\G(k,q)$ are generically non-bipartite. We first prove in Proposition \ref{prop caract. GP bip} that the only connected GP-graph which is bipartite is $\G(1,2)=K_2$. 
Then, combining Proposition~\ref{prop caract. GP bip} with Theorem~\ref{thm: components}, we prove that the GP-graphs $\G(k,q)$ are non-bipartite, with the only exception of the graphs 
	$$\G(2^m-1,2^m) \simeq K_2 \sqcup \cdots \sqcup K_2 \qquad (2^{m-1} \text{ copies})$$ 
with $m \in \N$ (see Theorem \ref{teo: GPbipartito}).
Using this result on non-bipartiteness of GP-graphs together with the already mentioned result of Brualdi et al, we completely characterize those connected bipartite doubles $B(\G(k,q))$ in Corollary \ref{coro: double}.
To conclude, in Example~\ref{exam: GP2s} we study the binary GP-graphs $\G(k,2^m)$ for $1\le m \le 8$,
giving their connected component decompositions.

\section{The connected components of the graphs $\G(k,q)$} \label{sec:2}
Suppose $q=p^m$, with $p$ prime and $m\in \N$, and let $k \in \N$ be such that $k\mid q-1$.
It is well-known that the generalized Paley graph $\G(k,q)$ is connected if and only if $n=\frac{q-1}{k}$ is a primitive divisor of $p^{m}-1$, denoted 
\begin{equation} \label{eq: prim div}
	n \dagger p^m-1,
\end{equation}	
(i.e.\@ $n\mid p^m-1$ but $n \nmid p^a-1$ for any $a<m$). 
In other words, $m$ is the order of $p$ modulo $n$, that is the least positive integer $t$ such that $p^t \equiv 1 \pmod n$, denoted 
\begin{equation} \label{eq: ord p}	
	ord_{n}(p)=m.
\end{equation}
We will use both notations alternatively in the paper (one emphasizes the relation between $n$ and $q$, 
when written $\frac{q-1}k \dagger q-1$, and the other between $p$ and $n$).

The fact that $\G(k,q)$ is connected is equivalent to the fact that the Waring number $g(k,q)$ exists. 
The \textit{Waring number} $g(k,q)$ is the minimum $s$ such that every $b \in \ff_q$ can be written as a sum of $s$ $k$-powers, i.e.\@ for each $b\in \ff_q$ there exist  $x_{1}, \ldots, x_{s} \in \ff_q$ such that 
	$b=x_{1}^k + \cdots + x_{s}^k$. 
In this case, the Waring number is given by the diameter of the graph (see for instance \cite{PV6} and \cite{PV7}), that is
\begin{equation} \label{eq: gkq}
	g(k,q) = \text{diam}(\G(k,q)).
\end{equation}

Notice that, if in the above notations we put $a:=ord_n(p)$, then $a\mid m$. Summing up, $\G(k,q)$ is connected if $a=m$ and not connected if $a<m$. As a result, $\G(k,p)$ is always connected for any $k\mid p-1$ and any $p$ prime 
(i.e.\@ $g(k,p)$ always exists).

Here we study the connected components of the graphs $\G=\G(k,q)$ in the most general case (i.e.\@ $\G$ directed or not). 
This was previously studied by Lim and Praeger in 2009 (\cite{LP}) in the undirected case where they showed that if $\G(k,q)$ is not connected then all their connected components are isomorphic to a smaller GP-graph defined over a subfield $\ff_{q'}$ of $\ff_{q}$. 
As a consequence, they obtained a decomposition for the automorphism group $\textrm{Aut}(\G)$ of $\G$.

Next, we extend this result to the general case, that is for $\G$ either directed or undirected.
We need first to recall the definition of induced graphs. Given $G=(V,E)$ any graph with vertex set $V$, edge set $E$, and $S\subseteq V$ any subset of vertices of $G$, 
the \textit{induced subgraph} in $G$ by $S$, denoted 
	$$G[S]=(S,E_S),$$ 
is the graph whose vertex set is $S$ and whose edge set consists of all of the edges in 
$E$ that have both endpoints in $S$. That is, for any two vertices $u,v \in S$, $u$ and $v$ are adjacent in $G[S]$ 
if and only if they are adjacent in $G$.

\begin{thm} \label{thm: components}
Let $q=p^m$ be a prime power with $m\in \N$, $k \in \N$ such that $k\mid q-1$ and put $n=\frac{q-1}{k}$. 
Let $a=ord_n(p)$ and let $k_a = \frac{p^{a}-1}n \in\N$. 
Then, $a \mid m$, $k_a \mid k$, and $\G=\G(k,q)$ has exactly $p^{m-a}$ (strongly) connected components, all mutually isomorphic to $\G_a = \G(k_a,p^a)$. That is, we have 
\begin{equation} \label{eq: comp cxas}
	\G(k,q) = \bigcup_{h \in \ff_q} \G[F_a+h] \simeq \underbrace{\G(k_a,p^a) \sqcup \cdots \sqcup \G(k_a,p^a)}_{\text{$p^{m-a}$ times}},
\end{equation} 
where $F_a = \{x \in \ff_q : x^{p^a}=x\}$.
Hence, the automorphism group of $\G$ is given by 
\begin{equation} \label{eq: Aut(G)}
	{\rm Aut}(\G) \simeq {\rm Aut}(\G_a) \wr \mathbb{S}_{p^{m-a}}
\end{equation} 
where $\wr$ denotes the wreath product of groups, $\mathbb{S}_{p^{m-a}}$ is the symmetric group on $p^{m-a}$ letters, and the spectrum of $\G$ is given by 
\begin{equation} \label{eq: specmults}
	Spec(\G) = [Spec(\G_a)]^{p^{m-a}}	
\end{equation}
meaning that the eigenvalues are the same, i.e.\@ $\{ \lambda_\G \} = \{ \lambda_{\G_a}\}$, but the multiplicities satisfy $m(\lambda_{\G}) = p^{m-a} m(\lambda_{\G_a})$.
\end{thm}

\begin{proof}
Since $n \mid p^{m}-1$ and $n\mid p^a-1$ with $a=ord_{n}(p)$, then $p^a-1 \mid p^m-1$ and hence we have that $a\mid m$.
Also, $k_a \mid k$ since $k=k_a \frac{p^m-1}{p^a-1}$.

Now, given $1\le a \le m$ with $a\mid m$, it is well-known that there is an isomorphic copy of $\ff_{p^a}$ inside $\ff_{q}=\ff_{p^m}$ given by the fixed elements of the Frobenius automorphism 
	$$\varphi_{a}(x)=x^{p^a}.$$
More precisely, we have
	$$F_a := \{ x\in \ff_{q}: \varphi_{a}(x)=x\}\cong \ff_{p^a}.$$

Furthermore, notice that if $\omega$ is a generator of $\ff_{q}^*$ and 
	$$\alpha= \omega^{\frac{q-1}{p^a-1}},$$ 
then $F_a = \langle \alpha \rangle$. 
Indeed,
	$$\varphi_{a}(\alpha) = \omega^{\frac{(q-1)p^a}{p^{a}-1}} = \omega^{\frac{(q-1)(p^a-1+1)}{p^{a}-1}} = \omega^{q-1}\omega^{\frac{q-1}{p^a-1}} = \alpha$$
and hence $\alpha \in F_a$ and $\langle \alpha \rangle \subset F_a$. Now, 
if $z\in F_a$ and $z\neq 0$, then $\varphi_{a}(z)=z$ and we can put $z=\omega^{\ell}$ for some $\ell$. In this way, we obtain that 
	$z^{p^a-1} = \omega^{\ell(p^{a}-1)}=1$.	
Thus, since $ord(\omega)=q-1$, there exist $t\in \mathbb{N}$ such that $\ell(p^{a}-1)=t(q-1)$ and hence 
	$z=\omega^{\ell} = \omega^{\frac{q-1}{p^{a}-1}t} = \alpha^{t}$.
Therefore, $F_{a} \subset \langle \alpha \rangle$ and hence 
	$$F_{a}=\langle \alpha \rangle.$$ 
	
The result will follow from three claims. 
In the first one we express the connection set $R_k$ of $\G(k,q)$ in terms of $F_a$.

\sk 	
\noindent \textit{$\bullet$ Claim $1$}: $R_k = \{w^{k_a} \in F_a : w\in F_{a}^* \}$.
\sk 	
	
Notice that $R_{k} = \langle \omega^k \rangle$ with $\omega$ a generator of $\ff_{q}^*$.
Since $k = \tfrac{q-1}{p^a-1}k_a$, 
we have that $R_{k}$ is generated by $\alpha^{k_a}$. 
Thus, we obtain that $R_{k}\subseteq \{w^{k_a} \in F_a : w\in F_{a}^* \}$. 
Finally, since 
	$$|R_{k}| = \tfrac{q-1}{k} = \tfrac{p^a-1}{k_a} = |\{w^{k_a} \in F_a : w\in F_{a}^* \}|,$$ 
we must have that $R_k = \{w^{k_a} \in F_a : w\in F_{a}^* \}$, as we claimed. \hfill $\lozenge$

\sk

Now, put $\G=\G(k,q)$. 
In the second claim we show that given a vertex $x$ of $\G(k,q)$, the set of neighbors of $x$, 
	$$N_{\G}(x) = \{ y \in \ff_{q}^* : xy \in E(\G)) \} = \{y \in \ff_q^* : y-x \in R_k\},$$ 
where $E(\G)$ is the set of (directed) edges of $\G$, is a subset of $F_a$.

\sk 
\noindent \textit{$\bullet$ Claim $2$}: $N_{\G}(x) \subseteq F_a$ for all $x\in F_a$.
\sk 
	
Let $x\in F_a$ and $y \in N_\G(x)$. We have to show that $y \in F_a$. Since $xy \in E(\G)$, there is $z \in R_k$ such that $y-x=z$, that is $y=x+z$. 
	By Claim $1$, $R_k\subseteq F_a$ and since $F_a$ is closed under addition (being a field), we have that $y=x+z\in F_a$ and so
	$N_{\G}(x)\subseteq F_a$. 
	\hfill $\lozenge$
	
\sk
 
Now, Claims 1 and 2 together imply the following:

\sk 	
\noindent \textit{$\bullet$ Claim $3$}: The induced subgraph of $\G$ by $F_a$, 
$\G[F_a] = (F_a,E(\G)_{F_a})$, 
is isomorphic to $\G(k_a,p^a) = Cay(\ff_{p^a},R_{k_a})$.
\sk 

First note that 
	$$\G[F_a]=Cay(F_a, R_k)$$ 
since $xy$ is an edge in $\G[F_a]$ if and only if $xy$ is an edge in $\G$, which happens if and only if $y-x \in R_k$. 
Hence, it is enough to show that there is a group isomorphism from $\ff_{p^a}$ to $F_a$ sending $R_{k_a}$ to $R_k$.
Let $\omega_a$ be a primitive element of $\ff_{p^a}$ and consider the field monomorphism 
	$$\psi_\alpha :\ff_{p^a} \hookrightarrow \ff_q$$  
defined by $\psi_\alpha(\omega_a) = \alpha$ and $\psi_\alpha(0)=0.$
Clearly, in this case we have that 
	$$\psi_\alpha(\ff_{p^a}) = F_{a} \qquad \text{and} \qquad \psi_\alpha(\langle \omega_a^{k_a}\rangle) = \langle \alpha^{k_a} \rangle =R_k,$$ 
by Claim 1.
Moreover, Claim $2$ implies that the regularity degree of all of the vertices in the induced subgraph of $\G[F_a]$ 
is exactly 
	$$|R_k| = |\{z^{k_a} \in \ff_{p^a} : z \in \ff_{p^a}^*\}|.$$
In this way, we see that $\psi_\alpha$ is surjective and hence a group isomorphism.
Therefore, 
	$$\G[F_a] \simeq \G(k_a,p^a),$$ 
as we claimed.
\hfill $\lozenge$

Since $a$ is minimal with the property $\frac{q-1}{k} \mid p^{a}-1$ and $\frac{q-1}{k} = \frac{p^a-1}{k_a}$, 
then $\frac{p^a-1}{k_a}$ is a primitive divisor of $p^{a}-1$ and hence $\G(k_a,p^a)$ is connected.
Hence $\G[F_a]$ is the connected component of $\G(k,q)$ containing the vertex $0$. 
The rest of the connected components are exactly 
	$\G[F_a + h]$ with $h\in \ff_q$ 
(i.e.\@ the induced subgraphs of $\G(k,q)$ generated by the cosets of $F_a$ in $\ff_{q}$).
Therefore, we obtain \eqref{eq: comp cxas} as desired.

Now, the expression for the automorphism group in \eqref{eq: Aut(G)} is automatic from the decomposition 
\eqref{eq: comp cxas} and the definition of wreath product. In fact, for the undirected case it was first proved in \cite{LP}. Since an automorphism of a directed graph is an automorphism preserving directed edges, the wreath product decomposition also holds for directed graphs.

Finally, since the spectrum of the disjoint union of graphs is the union of the spectrum of each of its components as multisets, the expression \eqref{eq: specmults} holds.
\end{proof}

We now make some comments on the previous result and some consequences.
\begin{rem} \label{rem: 1}
($i$) In (2) of Theorem 2.2 in \cite{LP}, Lim and Praeger showed the same statement as the previous theorem for the undirected case in a quite different way. Moreover, although not explicitly stated, their proof is also valid in the directed case, since it is purely algebraic. Nevertheless, for completeness, we gave an alternative proof.

\noindent ($ii$)
A graph is directed if and only if some of its connected components are directed. Notice that, by the decomposition \eqref{eq: comp cxas}, the GP-graph $\G(k,p^m)$ is directed if and only if $\G(k_a,p^a)$ is directed. Thus, by \eqref{eq: C1C2}, if $p$ is odd, to see if the graph is directed it is enough to check whether if $k \nmid \frac{p^a-1}2$ instead of checking the condition $k \nmid \frac{p^m-1}2$.

\noindent ($iii$)
Clearly, by the decomposition \eqref{eq: comp cxas}, those invariants of $\G(k,p^m)$ not depending on connectivities or transitivities of the graph (such as the girth, the chromatic number, the clique number, etc) can be computed from the smaller graph $\G(k_a,p^a)$. 

\noindent ($iv$)
Previously, in Lemma 4.1 of \cite{PV4b}, we have proved (in other notations) that for any $m$ even we have
\begin{equation} \label{eq: q^mKq^m}
	\G(q^{\frac m2}+1,q^m)=q^{\frac m2} K_{q^{\frac m2}}
\end{equation}
where $q=p^s$ and $s\in \N$. Notice that, taking $s=1$, this is just Theorem \ref{thm: components} with $a=\frac{m}2$.
\end{rem}
 
\noindent \textit{Notation:} Sometimes, for short, we will write $tH$ for the disjoint union of 
$t$ (isomorphic) copies of a graph $H$, although we prefer the long notation which is more visual. 
That is, if $H_1, \ldots, H_t$ are graphs isomorphic to $H$, then
\begin{equation} \label{eq: notation}
	tH := \underbrace{H \sqcup \cdots \sqcup H}_{\text{$t$-times}} \simeq H_1 \sqcup \cdots \sqcup H_t .
\end{equation}
Abusing the notation, we will use $=$ instead of $\simeq$.

We now give some basic examples of GP-graphs that can be decomposed as disjoint unions of cliques. 
In the first example the cliques are $K_2$.

\begin{exam} \label{ex: cliques}
($i$) Consider the GP-graph $\G(7,8)=Cay(\ff_8, \{x^7 : x \in \ff_8^*\})$. 
It is clear that $n=a=k_a=1$ in the notations of Theorem \ref{thm: components} and hence we have 
	$$\G(7,8) \simeq \G(1,2) \sqcup \G(1,2) \sqcup \G(1,2) \sqcup \G(1,2) =  4K_2.$$ 
Taking 
	$$\ff_8 = \ff_2[x]/(x^3+x+1) = \{0,1,\alpha,\alpha+1,\alpha^2,\alpha^2+1, \alpha^2+\alpha, \alpha^2+\alpha+1\}$$ 
with $\alpha^3+\alpha+1=0$, it is easy to see that the seventh powers of all these elements are $0$ and $1$ and hence the vertices 		
$\{0,1\}$, $\{\alpha,\alpha+1\}$, $\{\alpha^2,\alpha^2+1\}$ and $\{\alpha^2+\alpha,\alpha^2+\alpha+1\}$
form the four 2-cliques in $\G(7,8)$. 
They are precisely the induced graphs $\G[F_1]$, $\G[F_1+\alpha]$, $\G[F_1+\alpha^2]$, $\G[F_1+(\alpha^2+\alpha)]$, where $F_1=\{x\in \ff_8: x^2=x\}=\{0,1\}=\ff_2 \subset \ff_8$.

\noindent ($ii$) 
In general, for the GP-graph $\G(2^m-1,2^m) = Cay(\ff_{2^m}, \{x^{2^m-1} : x \in \ff_{2^m}^*\})$ we have that $n=a=k_a=1$ and hence 
\begin{equation} \label{eq: K2s}
	\G(2^m-1,2^m) \simeq 2^{m-1} K_2. 
\end{equation}
The 2-cliques are $\{\beta, \beta +1\}$ for any $\beta \in \ff_{2^m}$.
The automorphism group is 
	$$ {\rm Aut}(\G(2^m-1,2^m))=  \mathbb{Z}_2 \wr \mathbb{S}_{2^{m-1}}.$$
Also, by \eqref{eq: specmults} and \eqref{eq: K2s} we have that 
	$$ Spec(\G(2^m-1,2^m)) = [Spec(K_2)]^{2^{m-1}} = \big\{ [1]^{2^{m-1}}, [-1]^{2^{m-1}} \big\} $$
since the spectrum of the complete graph $K_\ell$ is 
\begin{equation} \label{eq: Spec Kt}
	Spec(K_\ell) = \{ [\ell-1]^1, [-1]^{\ell-1}\}
\end{equation}
and hence 
$Spec(K_2)=\{ [1]^1, [-1]^1 \}$.
\hfill $\lozenge$
\end{exam}

In the next example the cliques are $K_3=C_3$.

\begin{exam} \label{ex: cliquesK3}
($i$) Consider the GP-graph $\G(4,9)=Cay(\ff_9, \{x^4:x\in \ff_9^*\})$. Since $k=4$ and $q=9$ we have $n=2$, $a=1$ and $k_a=1$ in the notations of Theorem \ref{thm: components} and hence 
	$$\G(4,9) \simeq \G(1,3) \sqcup \G(1,3) \sqcup \G(1,3) = 3C_3. $$ 
Taking 
\begin{equation} \label{eq: F9}
	\ff_9 = \ff_3[x]/(x^2+1) = \{0,1,2,\alpha,\alpha+1,\alpha+2,2\alpha,2\alpha+1,2\alpha+2\}
\end{equation}	
with $\alpha^2+1=0$, it is easy to see that the fourth powers of all these elements are $0,1,2$ and hence the vertices 
$\{0,1,2\}$, $\{\alpha,\alpha+1,\alpha+2\}$ and $\{2\alpha,2\alpha+1,2\alpha+2\}$
form the 3-cycles in $\G(4,9)$. They are precisely the induced graphs $\G[F_1]$, $\G[F_1+\alpha]$, $\G[F_1+2\alpha]$, where $F_1=\{x\in \ff_9: x^3=x\}=\{0,1,2\}=\ff_3 \subset \ff_9$.

\noindent ($ii$)
In general, since $n=2$, $a=1$ and $k_a=1$, by Theorem \ref{thm: components} we have that 
\begin{equation*} \label{eq: K3s}
	\G(\tfrac{3^m-1}2,3^m) \simeq 3^{m-1}C_3. 
\end{equation*}
The automorphism group is 
	$$ {\rm Aut}(\G(\tfrac{3^m-1}2,3^m))=  \mathbb{D}_3 \wr \mathbb{S}_{3^{m-1}},$$
where $\mathbb{D}_3= \mathbb{S}_3$ is the dihedral group of 6 elements.
Furthermore, by \eqref{eq: specmults} and \eqref{eq: K2s} again, we have that 
	$$ Spec(\G(\tfrac{3^m-1}2,3^m)) = [Spec(C_3)]^{3^{m-1}} = \big\{ [2]^{3^{m-1}}, [-1]^{2\cdot 3^{m-1}} \big\} $$
since the spectrum of the $3$-cycle $C_3$ is $Spec(C_3) = \{ [2]^1, [-1]^{2}\}$.
\hfill $\lozenge$
\end{exam}

Finally, we give an example with cliques $K_{p^t}$ with $p$ prime. 
\begin{exam} \label{ex: cliques Kpt}
Now, consider $\G= \G(p^{\frac m2}+1,p^m) = Cay(\ff_{p^m}, \{x^{p^{m/2}+1} : x \in \ff_{p^m}^*\})$ for $m$ even. 
If we put $m=2t$ with $t\in \N$, by Remark \ref{rem: 1} $(iv)$ we have that  
\begin{equation} \label{eq: ptKpt}
	\G(p^t+1,p^{2t}) = p^t K_{p^t}.
\end{equation}
In this case we have $n=p^t-1$, $a=t$ and $k_a=1$. Thus, by Theorem \ref{thm: components}, we have 
$\G_a = \G(1,p^t) = K_{p^t}$ and hence the automorphism group and spectrum are given by 
\begin{gather*}
{\rm Aut}(\G(p^t+1, p^{2t})) = {\rm Aut}(K_{p^t}) \wr \mathbb{S}_{p^t} = \mathbb{S}_{p^t}\wr \mathbb{S}_{p^t}, \\
Spec(\G(p^t+1,p^{2t})) = [Spec(K_{p^t})]^{p^t} = \big\{ [p^t-1]^{p^t}, [-1]^{p^t(p^t-1)} \big\},
\end{gather*}	
where we have used \eqref{eq: Spec Kt}.	
\hfill $\lozenge$
\end{exam}

In the notations of Theorem \ref{thm: components}, the graph $\G(k,q)$ is connected if and only if $a=m$. 
We now give two families of connected GP-graphs, which turn out to be strongly regular. 
A \textit{strongly regular graph} with parameters $v,k,e,d$, denoted
	$$srg(v,k,e,d),$$
is a $k$-regular graph of order $v$ such that any pair of adjacent vertices has $e$ neighbors in common and any pair of non-adjacent vertices has $d$ common neighbors. The parameters satisfy 
\begin{equation} \label{eq: srg cond}
	(v-k-1)d=k(k-e-1).	
\end{equation}
If $d\ne 0$, the graph is connected with 3 eigenvalues and has diameter 2.

\begin{exam} \label{ex: qml}
($i$) In \cite{PV4b}, we studied the family of GP-graphs $\G(q^\ell+1, q^m)$ with $\ell \mid m$, where $q=p^s$ and $p$ prime. 
For $\frac{m}{\ell}$ even we get non-trivial graphs (if $\frac{m}{\ell}$ is odd we get either the complete graph $K_q$ or the Paley graph $\mathcal{P}_q$). All these graphs are connected (for $\ell \ne \frac m2$ if $m$ is even) and strongly regular (hence with 3 different eigenvalues). The spectra of $K_q$ and $P_q$ are well-known and their parameters 
as strongly regular graph are 
	$$K_q =srg(q,q-1,q-2,0) \qquad \text{and} \qquad \mathcal{P}_q = srg(q, \tfrac{q-1}2, \tfrac{q-5}4, \tfrac{q-1}4).$$

For the graphs with $\frac{m}{\ell}$ even, their eigenvalues are given in Theorem  3.5 in \cite{PV4b} and their parameters as strongly regular graph in Theorem 5.2 in \cite{PV4b}.  
In the present notation, for any pair of numbers $\ell, t \in \N$ with $t>1$, we have the graphs
\begin{equation} \label{eq: GPqml}
	\G(p^\ell+1,p^m) = srg(p^m, \tfrac{p^m-1}{p^\ell+1}, \tfrac{p^m -(-1)^t p^{\ell(t+1)} (p^\ell-1) -3p^\ell-2}{(p^\ell+1)^2},
	\tfrac{p^m +(-1)^t p^{\ell t} (p^\ell-1) -p^\ell}{(p^\ell+1)^2})
\end{equation}
where $m=2\ell t$ and $\ell \ne \frac m2$.  
This example complements the study of the graphs considered in Example \ref{ex: cliques Kpt} where $\ell=\frac m2$.

\noindent ($ii$) 
For $q=p^m$ with $m \in \N$ and $p$ an odd prime we have  
\begin{equation} \label{eq: rook}
	\G(\tfrac{q+1}2, q^2) = L_{q,q} = srg(q^2,2(q-1),q-2,2)
\end{equation}
where $L_{q,q}$ is the lattice (or rook's graph). In Proposition 5.7 in \cite{PV3} we showed the first equality; the second one is well-known (also that $\G(\tfrac{q+1}2, q^2)$ is the only Hamming GP-graph which is semiprimitive).
\hfill $\lozenge$
\end{exam}

\section{GP-graphs as disjoint unions of cycles} \label{sec:3}
We now characterize those GP-graphs which are disjoint unions of cycles (directed or not).
This will be used in the next section to classify bipartiteness of GP-graphs.

We first characterize those connected GP-graphs which are cycles of prime length.
\begin{lem} \label{lema: cycles}
If $p$ is an odd prime then we have  
	\begin{equation} \label{eq: Cp and Cp*}
	\G(\tfrac{p-1}2,p) \simeq C_p \qquad \text{and} \qquad \G(p-1,p) \simeq \vec{C}_p,
	\end{equation}	
where $C_p$ denotes the $p$-cycle graph and $\vec{C}_p$ the oriented 
$p$-cycle graph.
\end{lem}

\begin{proof}
Since $p$ is prime, by (C1) in \eqref{eq: C1C2}, the graph $\G(k,p)$ is connected for any $k$, hence in particular for $k=\frac{p-1}2$ and $k=p-1$. 
Also, by condition (C2) in \eqref{eq: C1C2}, the GP-graph	
	$\G(\frac{p-1}2,p)$ is undirected since $\frac{p-1}2$ is trivially a divisor of $\frac{p-1}2$ and $\G(p-1,p)$ is directed since $p-1$ is not a divisor of $\frac{p-1}2$.
	
Now, $\G(\frac{p-1}2,p)$ is $n$-regular with $n=\frac{p-1}k$ and $k=\frac{p-1}2$, hence $2$-regular. Since the graph is undirected, $R_{(p-1)/2}$ is symmetric and hence $-1 \in R_{(p-1)/2}$. Thus, any pair of consecutive vertices $a, a+1$ are connected and thus, $\G(\frac{p-1}2,p) \simeq C_p$.
	
For the directed graph $\G(p-1,p)$, there is an oriented
edge from any vertex $a$ to $a+1$ since $(a+1)-a=1\in R_{p-1}$, but not from $a+1$ to $a$ since $-1 \not\in R_{p-1}$. This is because by Fermat's little theorem we have that 
	$$a^{p-1}\equiv 1 \pmod p$$ 
for any $a$ coprime with $p$. Precisely this theorem implies that these are the only arcs in the digraph.
This proves $\G(p-1,p) \simeq \vec{C}_p$ and the result follows.
\end{proof}

We can now give the characterization of GP-graphs as a disjoint union of cycles. 

\begin{prop} \label{prop: cycle decompositions}
Let $p$ be an odd prime and $q=p^m$ with $m\in \N$. Then, we have the connected component decompositions
\begin{equation} \label{eq: Cps and Cp*s}
	\G(\tfrac{q-1}2,q) \simeq \underbrace{C_p \sqcup \cdots \sqcup C_p}_{p^{m-1} \text{times}} \qquad \text{and} \qquad \G(q-1,q) \simeq \underbrace{\vec{C}_p \sqcup \cdots \sqcup \vec{C}_p}_{p^{m-1} \text{times}}
\end{equation}	
where the (directed) $p$-cycles are the cosets $\omega^i + \Z_p$ of the (directed) $p$-cycle $\Z_p = \{0,1,\ldots,p-1\}$ \nolinebreak with $\omega$ a primitive element of $\ff_q^*$.
Their automorphism groups are 
\begin{equation} \label{eq: Aut Cps and Cp*s}
	{\rm Aut}(\G(\tfrac{q-1}2,q)) \simeq \mathbb{D}_p \wr \mathbb{S}_{p^{m-1}} \qquad \text{and} \qquad 
	{\rm Aut}(\G(q-1,q)) \simeq \mathbb{Z}_p \wr \mathbb{S}_{p^{m-1}},
\end{equation}
and their spectra are given by 
\begin{equation} 
  \begin{aligned} \label{eq: Spec Cps and Cp*s}	
	Spec(\G(\tfrac{q-1}2,q)) & = \{ [2\cos(\tfrac{2\pi j}p)]^{p^{m-1}} \}_{0 \le j \le p-1}, \\ 
	Spec(\G({q-1},q))        & = \{ [e^{\frac{2\pi j}p}]^{p^{m-1}} \}_{0 \le j \le p-1}.
  \end{aligned}
\end{equation}

Conversely, if the graph $\G(k,q)$ with $k\mid q-1$ and $q=p^m$ with $p$ an odd prime and $m\in \N$ has connected component decomposition as a union of cycles, then $\G(k,q)$ is as one of the graphs in \eqref{eq: Cps and Cp*s}.
\end{prop}

\begin{proof}
For the graphs 
\begin{equation} \label{eq: Cay q-1}
	\G(\tfrac{p^m-1}2,p^m) = Cay(\ff_{p^m}, \{x^{\frac{p^m-1}2} : x\in \ff_{p^m}^*\}) = Cay(\ff_{p^m}, \{ \pm 1\})
\end{equation}
with $p$ odd we have $n=2$, $a=1$ and $k_a=\frac{p-1}2$ in the notations of Theorem \ref{thm: components}. Hence, we obtain
the disjoint union 
	$$\G(\tfrac{p^m-1}2,p^m) \simeq \G(\tfrac{p-1}2,p) \sqcup \cdots \sqcup \G(\tfrac{p-1}2,p) \simeq C_p \sqcup \cdots \sqcup C_p$$
with $p^{m-1}$ components, where we have used Lemma \ref{lema: cycles}.

Similarly, for the graphs 
\begin{equation} \label{eq: Cay q-1/2}
	\G(p^m-1,p^m) = Cay(\ff_{p^m}, \{x^{p^m-1} : x\in \ff_{p^m}^*\}) = Cay(\ff_{p^m}, \{ 1\})
\end{equation}
with $p$ odd we have $n=2$, $a=1$ and $k_a=p-1$, and thus we get the disjoint union 
	$$\G(p^m-1,p^m) \simeq \G(p-1,p) \sqcup \cdots \sqcup \G(p-1,p) \simeq \vec{C}_p \sqcup \cdots \sqcup \vec{C}_p$$
with $p^{m-1}$ components, where we have used Lemma \ref{lema: cycles} again. This proves \eqref{eq: Cps and Cp*s}. 

The statement on the (directed) $p$-cycles realized as cosets is a consequence of Theorem~\ref{thm: components}, while the assertion on the automorphisms groups is automatic from Theorem~\ref{thm: components} and the basic fact that 
\begin{equation} \label{eq: AutCps}
	\textrm{Aut}(C_p) = \mathbb{D}_p \qquad \text{ and } \qquad \textrm{Aut}(\vec{C}_p)=\mathbb{Z}_p.
\end{equation}

The expressions for the spectra are direct consequences of \eqref{eq: specmults} and the fact that 
\begin{gather*}
Spec(\G(\tfrac{p-1}2,p)) = Spec(C_p)= \{2, 2\cos(\tfrac{2\pi}p), 2\cos(\tfrac{2\cdot 2\pi}p), \ldots, 2\cos(\tfrac{2(p-1)\pi}p) \}, \\
Spec(\G(p-1,p)) = Spec(\vec{C}_p) = \{1, \zeta_p, \zeta_p^2, \ldots, \zeta_p^{p-1}\},
\end{gather*}	
where $\zeta_p=e^{\frac{2\pi i}p}$, and the proof is complete. 
\end{proof}

We now give some comments and consequences of the previous result.

\begin{rem} \label{rem: dir/undir}
($i$) 
Observe that from \eqref{eq: AutCps} in the above proof we have that 
	$$ {\rm Aut}(C_p) = \Z_2 \rtimes  {\rm Aut}(\vec{C}_p)  \qquad \text{ and } \qquad 
		Spec(C_p) = 2 Re(Spec(\vec{C}_p)).$$
This can be explained by the fact that one usually considers Cayley graphs $Cay(G,S)$ to be undirected when in fact they have by definition all the edges with the two possible arrows (i.e.\@ one considers oriented 2-cycles as undirected edges, see the Introduction). In the case of the cycle $C_p$, we are considering it as the union of the directed cycle and the directed cycle with the arrows reversed, that is 
	$$C_p = \overset{_{\rightarrow}}{C}_p \cup \overset{_{\leftarrow}}{C}_p.$$ 		
In particular, we can think $\G(p-1,p)=\vec{C}_p$ as a subgraph of $\G(\frac{p-1}2,p)=C_p$.

\noindent ($ii$)
More generally, from the expressions \eqref{eq: Cps and Cp*s} and \eqref{eq: Spec Cps and Cp*s} one notices that  
	$$Spec(\G(\tfrac{q-1}2,q)) = 2 Re(Spec(\G(q-1,q))).$$
Indeed, recall that the eigenvalues of $Cay(G,S)$ are given by the sums 
	$$\chi(S)=\sum_{s\in S} \chi(s),$$ 
for characters $\chi \in \widehat{G}$.
Thus, since $R_{\frac{q-1}{2}} = R_{q-1} \cup (-R_{q-1})$ then for any $a\in \ff_{q}$ we have
	$$\chi_{a}(R_{\frac{q-1}{2}})= \chi_{a}(R_{q-1})+\chi_{a}(-R_{q-1})=2 Re(\chi_{a}(R_{q-1}))$$
where $\chi_{a}$ is the additive character of $\ff_{q}$ corresponding to $a$, i.e.\@ $\chi_{a}(x)=e^{\frac{2\pi i}{p}\Tr_{q/p}(ax)}$.
\end{rem}
Relative to Waring numbers (see the beginning of Section \ref{sec:2}) we have the following.

\begin{rem}
It is known that for $p$ prime, $g(\frac{p-1}2,p)=\frac{p-1}2$ and $g(p-1,p) = p-1$. 
By \eqref{eq: gkq} and Lemma \ref{lema: cycles} we recover this classic result. Furthermore, from \eqref{eq: gkq} and 
Proposition~\ref{prop: cycle decompositions} we obtain that the Waring numbers $g(\frac{p^m-1}2,p^m)$ and $g(p^m-1,p^m)$ do not exist for every $m>1$, since the graphs $\G(\frac{p^m-1}2, p^m)$ and $\G(p^m-1, p^m)$ are disconnected (see the comments before \eqref{eq: gkq}). Moreover, by \eqref{eq: q^mKq^m} the numbers $g(q^m+1,q^{2m})$ do not exist for every $m\in \N$ and by \eqref{eq: GPqml} and \eqref{eq: rook} we recover the known results 
	$g(p^m+1, p^{2mt})=g(q^m+1,q^{2m})=2$ 
for any $m,t\in \N$ with $t>1$.
\end{rem}

For $q=p^m$ odd, we always have four granted GP-graphs: 
\begin{enumerate}[$(i)$]
	\item $\G(1,q)=K_q$ (undirected, connected, $(q-1)$-regular),   
	
	\item $\G(2,q)= \mathcal{P}_q$ or $\vec{\mathcal{P}}_q$ (undirected if $q\equiv 1 \pmod 4$ or directed if $q\equiv 3 \pmod 4$, connected, $(\frac{q-1}2)$-regular), 
	
	\item $\G(\frac{q-1}2,q)=p^{m-1}C_p$ (undirected, disconnected, $2$-regular), and 
	
	\item $\G(q-1,q) = p^{m-1} \vec{C}_p$ (directed, disconnected, $1$-regular). 
\end{enumerate} 
That is, we always have
\begin{equation} \label{eq: granted}
 \G(q-1,q) \subset \G(\tfrac{q-1}2,q) \subset \G(1,q) 
	\qquad \text{and} \qquad  
	\G(q-1,q) \subset \G(2,q) \subset \G(1,q).
\end{equation}	
If $q \equiv 1 \pmod 4$, then also $\G(\tfrac{q-1}2,q) \subset \G(2,q)$ 
and we have the chain of subgraphs 
	$$\G(q-1,q) \subset \G(\tfrac{q-1}2,q) \subset \G(2,q) \subset \G(1,q).$$ 
In addition, over finite fields of square cardinality, we have the GP-graphs: 
$\G(q+1,q^2) = qK_q$ by \eqref{eq: ptKpt} and $\G(\tfrac{q+1}2,q^2) = L_{q,q}$ by \eqref{eq: rook}. 

Not only that, for each $k\mid p^m-1$, the graph $\G=\G(k,p^m)$ induces the subfield subgraph $\G(k,p^t)$ for each divisor 
$t \mid m$. If we want a subgraph of the same order of $\G$ we just take the disjoint union of $p^{m-t}$ copies of  
$\G(k,p^t)$, that is 
\begin{equation} \label{eq: subgraphs}
	p^{m-t}\G(k,p^t) \subset \G(k,p^m).
\end{equation}
For a fixed finite field $\ff_q$ with $q=p^m$, it is interesting to study the lattice of all the subgraphs $\G(k,p^t)$ with $t\mid m$ and $k\mid q-1$ of $\G(1,q)=K_q$.

\subsubsection*{Worked examples}
Now, we give the GP-graphs $\G(k,p^m)$ for the first odd primes $p=3,5,7$ and the smallest values of $m$. 
We will use \eqref{eq: C1C2}, Theorem \ref{thm: components} and Proposition \ref{prop: cycle decompositions}. 
We use conditions in \eqref{eq: C1C2} to determine directedness and connectedness. 
If the graph is disconnected, we give its connected components. 
When possible, we describe the GP-graphs in terms of known graphs (complete, Paley, cycle, lattice and strongly regular graphs). 

\begin{exam}[$p=3$, $1\le m \le 4$] \label{exam: GP3s}
Here we give the GP-graphs $\G(k,3^m)$ with $k\mid 3^m-1$ and $1\le m \le 4$. 
For $m=1$, we have the graphs 
	$$\G(1,3)=K_3=C_3 \qquad \text{and} \qquad \G(2,3) =  \vec{\mathcal{P}}_3 = \vec{C}_3.$$ 
Notice that here the 4 granted graphs mentioned in \eqref{eq: granted} coincide in pairs.

For $m=2$, we obtain the graphs: 
\begin{alignat}{2} \label{eq: Gk9}
\begin{aligned}
& \G(1,9) = K_9, && \qquad \G(2,9) = \mathcal{P}_9 = L_{3,3}, \\ 
& \G(4,9) = 3C_3 = 3K_3 = 3\G(1,3), && \qquad \G(8,9) = 3\vec{C}_3 = 3 \vec{\mathcal{P}}_3 = 3\G(2,3).  
\end{aligned}
\end{alignat}
These graphs are depicted below, where we use the labeling of the vertices as in \eqref{eq: F9}.
Notice that $\alpha^2+1=0$ and, since we are in characteristic 3, $\alpha^2=-1=2$. Hence, the non-zero squares, fourth powers and eighth powers in $\ff_9$ are given by 
	$$\{x^2\}_{x\in \ff_9^*} =\{1,2,\alpha, \alpha^2\}, \qquad \{x^4\}_{x\in \ff_9^*} = \{1,2\} \qquad \text{and} \qquad 
	\{x^8\}_{x\in \ff_9^*} = \{1\}.$$ 
Using this, it is easy to see that: 

\noindent

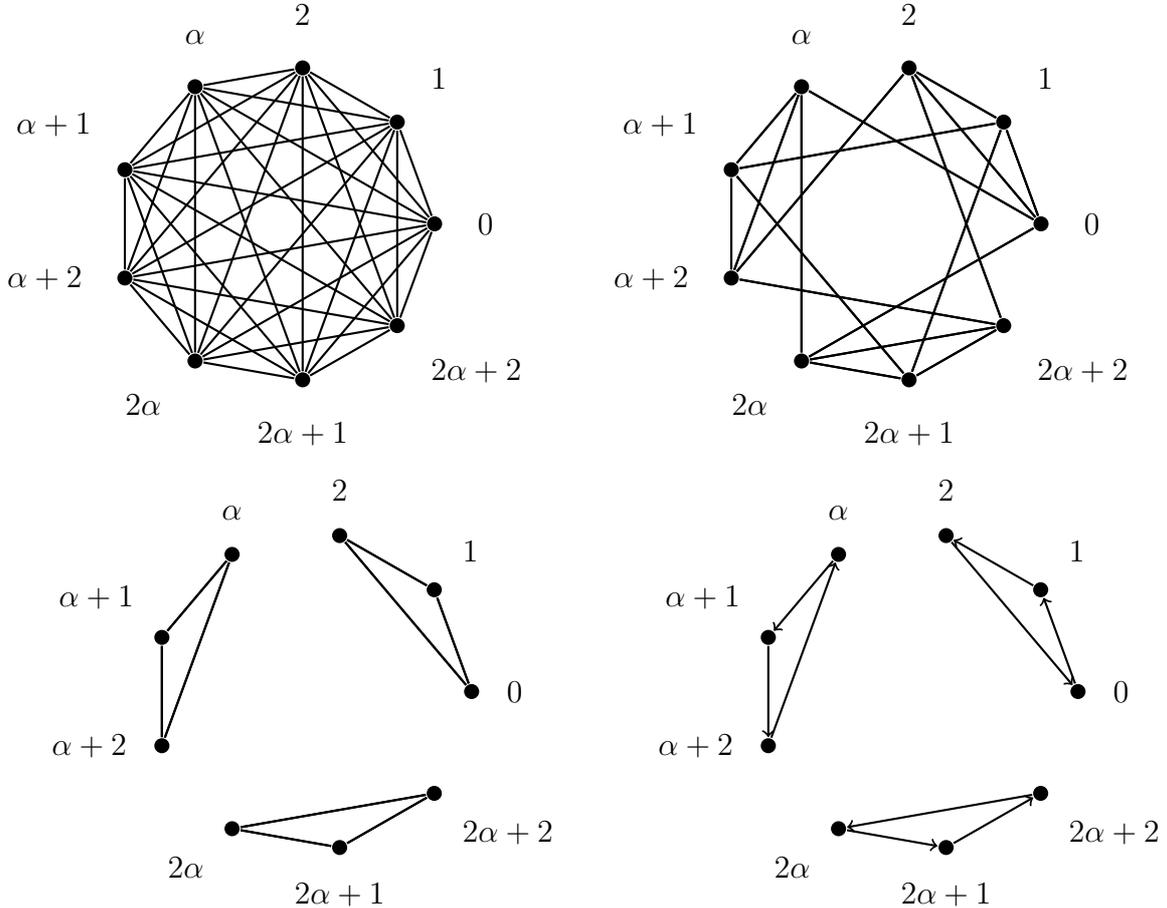
\begin{figure}[h!]
\begin{minipage}{0.45\textwidth}
\centering
\begin{tikzpicture}[scale=0.85, thick, 
	main node/.style={fill=black, circle, inner sep=2pt}, 
	label distance=3mm] 
	\foreach \i/\name/\pos in {
		0/$0$/right, 
		1/$1$/above right, 
		2/$2$/above, 
		3/$\alpha$/above, 
		4/$\alpha+1$/above left, 
		5/$\alpha+2$/left, 
		6/$2\alpha$/below left, 
		7/$2\alpha+1$/below , 
		8/$2\alpha+2$/below right} 
	{
		\node[main node] (n\i) at ({360/9 * \i}:3cm) [label=\pos:\name] {}; 
	}
	
	\foreach \i in {0,1,2,3,4,5,6,7,8} {
		\foreach \j in {0,1,2,3,4,5,6,7,8} {
			\ifnum \i<\j
			\path[thick, draw=black] (n\i) edge (n\j);
			\fi
		}
	}
	
\end{tikzpicture}
\end{minipage}
\hspace{.5cm}  
\begin{minipage}{0.45\textwidth}
	\centering
                                 	
\begin{tikzpicture}[scale=0.85, auto, 
thick, main node/.style={fill=black, circle, inner sep=2pt}, 
	label distance=3mm] 
	
	\foreach \i/\name/\pos in {
		0/$0$/right, 
		1/$1$/above right, 
		2/$2$/above, 
		3/$\alpha$/above, 
		4/$\alpha+1$/above left, 
		5/$\alpha+2$/left, 
		6/$2\alpha$/below left, 
		7/$2\alpha+1$/below , 
		8/$2\alpha+2$/below right} 
	{
		\node[main node] (n\i) at ({360/9 * \i}:3cm) [label=\pos:\name] {}; 
	}
	
	
	
	\path[thick, draw=black] (n0) edge (n1); 
	\path[thick, draw=black] (n0) edge (n2); 
	\path[thick, draw=black] (n0) edge (n3); 
	\path[thick, draw=black] (n0) edge (n6); 
	
	\path[thick, draw=black] (n1) edge (n0); 
	\path[thick, draw=black] (n1) edge (n2); 
	\path[thick, draw=black] (n1) edge (n4); 
	\path[thick, draw=black] (n1) edge (n7); 
	
	\path[thick, draw=black] (n2) edge (n0); 
	\path[thick, draw=black] (n2) edge (n1); 
	\path[thick, draw=black] (n2) edge (n5); 
	\path[thick, draw=black] (n2) edge (n8); 
	
	\path[thick, draw=black] (n3) edge (n0); 
	\path[thick, draw=black] (n3) edge (n4); 
	\path[thick, draw=black] (n3) edge (n5); 
	\path[thick, draw=black] (n3) edge (n6); 
	
	\path[thick, draw=black] (n4) edge (n1); 
	\path[thick, draw=black] (n4) edge (n3); 
	\path[thick, draw=black] (n4) edge (n5); 
	\path[thick, draw=black] (n4) edge (n7); 
	
	\path[thick, draw=black] (n5) edge (n2); 
	\path[thick, draw=black] (n5) edge (n3); 
	\path[thick, draw=black] (n5) edge (n4); 
	\path[thick, draw=black] (n5) edge (n8); 
	
	\path[thick, draw=black] (n6) edge (n0); 
	\path[thick, draw=black] (n6) edge (n3); 
	\path[thick, draw=black] (n6) edge (n7); 
	\path[thick, draw=black] (n6) edge (n8); 
	
	\path[thick, draw=black] (n7) edge (n1); 
	\path[thick, draw=black] (n7) edge (n4); 
	\path[thick, draw=black] (n7) edge (n8); 
	\path[thick, draw=black] (n7) edge (n6); 
	
	\path[thick, draw=black] (n8) edge (n2); 
	\path[thick, draw=black] (n8) edge (n5); 
	\path[thick, draw=black] (n8) edge (n7); 
	\path[thick, draw=black] (n8) edge (n6); 
	
\end{tikzpicture}
\end{minipage}
\caption{The graphs $\G(1,9)$ and $\G(2,9)$.}
\end{figure}
                                 	
\begin{figure}
\begin{minipage}{0.45\textwidth}
	\centering
\begin{tikzpicture}[scale=0.85, auto, thick, 
	main node/.style={fill=black, circle, inner sep=2pt}, 
	label distance=2mm] 
	
	\foreach \i/\name/\pos in {
		0/$0$/right, 
		1/$1$/above right, 
		2/$2$/above, 
		3/$\alpha$/above, 
		4/$\alpha+1$/above left, 
		5/$\alpha+2$/left, 
		6/$2\alpha$/below left, 
		7/$2\alpha+1$/below , 
		8/$2\alpha+2$/below right} 
	{
		\node[main node] (n\i) at ({360/9 * \i}:3cm) [label=\pos:\name] {}; 
	}
	
	\path[thick, draw=black] (n0) edge (n1); 
	\path[thick, draw=black] (n0) edge (n2); 
	
	\path[thick, draw=black] (n1) edge (n0); 
	\path[thick, draw=black] (n1) edge (n2); 
	
	\path[thick, draw=black] (n2) edge (n0); 
	\path[thick, draw=black] (n2) edge (n1); 
	
	\path[thick, draw=black] (n3) edge (n4); 
	\path[thick, draw=black] (n3) edge (n5); 
	
	\path[thick, draw=black] (n4) edge (n3); 
	\path[thick, draw=black] (n4) edge (n5); 
	
	\path[thick, draw=black] (n5) edge (n3); 
	\path[thick, draw=black] (n5) edge (n4); 
	
	\path[thick, draw=black] (n6) edge (n7); 
	\path[thick, draw=black] (n6) edge (n8); 
	
	\path[thick, draw=black] (n7) edge (n6); 
	\path[thick, draw=black] (n7) edge (n8); 
	
	\path[thick, draw=black] (n8) edge (n6); 
	\path[thick, draw=black] (n8) edge (n7); 
	
\end{tikzpicture}
\end{minipage}
\hspace{0.5cm} 
\begin{minipage}{0.45\textwidth}
\centering
\begin{tikzpicture}[scale=0.85, auto, thick, 
	main node/.style={fill=black, circle, inner sep=2pt}, 
	label distance=2mm] 

\foreach \i/\name/\pos in {
	0/$0$/right, 
	1/$1$/above right, 
	2/$2$/above, 
	3/$\alpha$/above, 
	4/$\alpha+1$/above left, 
	5/$\alpha+2$/left, 
	6/$2\alpha$/below left, 
	7/$2\alpha+1$/below , 
	8/$2\alpha+2$/below right} 
{
	\node[main node] (n\i) at ({360/9 * \i}:3cm) [label=\pos:\name] {}; 
}
	
	\path[->, draw=black] (n0) edge (n1); 
	
	\path[->, draw=black] (n1) edge (n2); 
	
	\path[->, draw=black] (n2) edge (n0); 
	
	\path[->, draw=black] (n3) edge (n4); 
	
	\path[->, draw=black] (n4) edge (n5); 
	
	\path[->, draw=black] (n5) edge (n3); 
	
	\path[->, draw=black] (n6) edge (n7); 
	
	\path[->, draw=black] (n7) edge (n8); 
	
	\path[->, draw=black] (n8) edge (n6); 
\end{tikzpicture}
\end{minipage}
\caption{The graphs $\G(4,9)$ and $\G(8,9)$.}
\end{figure}

For $m=3$, we have the GP-graphs:
\begin{alignat*}{2}
& \G(1,27) = K_{27}, && \qquad \G(2,27) = \vec{\mathcal{P}}_{27}, \\ 
&\G(13,27) = 9 C_3 = 9 K_3 =9 \Gamma(1,3), && \qquad \G(26,27) = 9 \vec{C}_3 = 9 \vec{\mathcal{P}}_3 = 9 \G(2,3).
\end{alignat*}	

Finally, for $m=4$, considering the divisors of $3^4-1=80$ we have 
the connected components decomposition of the 
GP-graphs $\G(k,81)$ for the divisors $k=1,2,4,5,8,10,16,20,40,80$: 
\begin{equation} \label{eq: G81's)}
 \begin{aligned}
		\G(1,81)  & = K_{81}, \\ 
		\G(2,81)  & = \mathcal{P}_{81}, \\ 
		\G(4,81)  & = srg(81,20,1,6) = \text{Brouwer-Haemers graph}, \\ 
		\G(5,81)  & = srg(81,16,7,2) = L_{9,9}, \\ 
		\G(8,81)  & = \text{undirected, connected, $10$-regular (not srg)}, \\ 
		\G(10,81) & = 9K_9, \\
		\G(16,81) & = \text{directed, connected, $5$-regular (not srg)}, \\ 
		\G(20,81) & = 9 \mathcal{P}_9, \\
		\G(40,81) & = 27 C_3 = 27 K_3, \\ 
		\G(80,81) & = 27 \vec{C}_3 = 27 \vec{\mathcal{P}}_3. 
 \end{aligned}
\end{equation}
These graphs are undirected for $k\ne 16, 80$, since $k\mid \frac{3^4-1}2=40$ in this case.
To see connectedness, for instance, for $\G(20,81)$ we have 
$n=\tfrac{81-1}{20}=4$ and $3^2\equiv 1 \pmod 4$ and thus $a=2$.  
Hence, $k_a=\frac{3^2-1}{4}=2$ and thus by Theorem \ref{thm: components}, the graph has $3^{4-2}=9$ connected components isomorphic to $\G(k_a,3^2)=\G(2,9)=\mathcal{P}_9$. 

Now, notice that 
	$$\G(4,81) = \G(3^1+1,3^4)$$ 
is as in \eqref{eq: GPqml}, hence a strongly regular graph with parameters $srg(81,20,1,6)$. This is the Brouwer-Haemers graph which is known to be Ramanujan and unique with these parameters. 
Also, 
	$$\G(10,81)=\G(3^2+1, 3^4)$$
is as in \eqref{eq: ptKpt} and hence $9$ copies of $K_9$.
Notice that 
	$$\G(5,81)=\G(\tfrac{3^2+1}2,3^4),$$ 
and hence by \eqref{eq: rook} it is the rook's graph $L_{9,9}=srg(81,16,7,2)$. 
The graph $\G(8,81)$ is not strongly regular, for if it were, its parameters $srg(81,8,e,d)$ together with \eqref{eq: srg cond} would imply that $e=2$ and $d=0$, but $srg(81,16,2,0)$ does not exist (see Brouwer's tables of strongly regular graphs \cite{BWP}).  
Finally, $\G(16,81)$ is not strongly regular since it is directed.

Relative to the comments around \eqref{eq: granted} and \eqref{eq: subgraphs}, observe that for every $k \mid 80$, $k' \mid 8$ and $k'' \mid 2$, we have the GP-graphs 
	$$\G(k,3^4), \qquad 9 \G(k',3^2) \qquad  \text{and} \qquad 27 \G(k'',3)$$ 
of 81 vertices. This gives $16=10+4+2$ GP-graphs, but we know that there are only 10 GP-graphs of order 81. 
Hence, there are some repetitions. 
In the previous list in \eqref{eq: G81's)} we see that 14 GP-graphs appear in the connected component decompositions.
For instance, for $k=1,2$ we have the graphs  
$\G(1,81)=K_{81}$, $9\G(1,9)=9K_{9}$, $27\G(1,3)=27K_3$ 
and 
$\G(2,81)=\mathcal{P}_{81}$, $9\G(2,9)=9\mathcal{P}_9$, $27\G(2,3)=27\mathcal{P}_{3}$. 
For $k=4$ we get 
$$\G(4,81), \qquad 9\G(4,9) = 9(3K_3) \qquad \text{and} \qquad 27\G(4,3)=27\G(1,3)=27K_3,$$
where we used \eqref{eq: Gk9}. 
The remaining two not used graphs $\G(4,9)$ and $\G(8,9)$ 
give interesting decompositions as disjoint unions, which are not 
the connected component decompositions. 
Namely, we get a disjoint union decomposition of the graph 
	$$ \G(40,81)=9\G(4,9)=9(3K_3),$$ 
which is not the connected component decomposition. Similarly, for $k=8$, the graph $9\G(8,9)=9(3\vec{C}_3)$ gives another disjoint union decomposition of $\G(80,81)$ which is not the connected component decomposition.
\hfill $\lozenge$
\end{exam}

\begin{exam}[$p=5$]
Here we give the GP-graphs $\G(k,5^m)$ with $k\mid 5^m-1$ and $m=1,2$.
For $m=1$, we have the graphs 
	$$\G(1,5)=K_5, \qquad \G(2,5) = \mathcal{P}_5 = C_5 \qquad \text{ and } \qquad \G(4,5)=\vec{C}_5.$$ 
Notice that here $\frac{5-1}2=2$, so $\G(2,q)=\G(\frac{q-1}2,q)$.

For $m=2$, the divisors of $24$ are $1,2,3,4,6,8,12,24$ and hence we have the eight 
graphs:
\begin{align*}
\G(1,25)  & = K_{25}, \\ 
\G(2,25)  & = \mathcal{P}_{25}, \\ 
\G(3,25)  & = srg(25,8,3,2) = L_{5,5}, \\ 
\G(4,25)  & = \text{undirected, connected, $6$-regular (not srg)}, \\ 
\G(6,25)  & = K_5 \sqcup K_5 \sqcup K_5 \sqcup K_5 \sqcup K_5, \\ 
\G(8,25)  & = \text{directed, connected, $3$-regular (not srg)}, \\
\G(12,25) & = C_5 \sqcup C_5 \sqcup C_5 \sqcup C_5 \sqcup C_5 = 
 	 \mathcal{P}_5 \sqcup \mathcal{P}_5 \sqcup \mathcal{P}_5 \sqcup \mathcal{P}_5 \sqcup \mathcal{P}_5, \\ 
\G(24,25) & = \vec{C}_5 \sqcup \vec{C}_5 \sqcup \vec{C}_5 \sqcup \vec{C}_5 \sqcup \vec{C}_5. 
\end{align*}
Since $k\mid \frac{q-1}2 = 12$ for any $k \ne 8, 24$ all the graphs are undirected except for $\G(8,25)$ and $\G(24,25)$. 
We know the extreme cases, those for $k=1,2,12,24$. So, we now explain the others. 
For instance, for $\G(4,25)$ we have 
$n=\tfrac{25-1}4=6$ and $5^2\equiv 1 \pmod 6$, hence $a=2=m$ and  
the graph is connected.
For $\G(6,25)$, $n=\frac{25-1}6=4$ and $5^1\equiv 1 \pmod 4$, hence $a=1$ and $k_a=\frac{5^a-1}{4}=1$.  Thus, by Theorem \ref{thm: components}, the graph has $5^{2-1}$ connected components isomorphic to $\G(k_a,5^a)=\G(1,5)=K_5$. 
 Similarly for the others.

Notice that 
	$$\G(3,25) = \G(\tfrac{5^1+1}2,5^2)$$ 
and hence by \eqref{eq: rook} it is the rook's graph $L_{5,5}=srg(25,8,3,2)$. 
The 6-regular graph $\G(4,25)$ is not strongly regular since condition \eqref{eq: srg cond} on the parameters $(25,6,e,d)$ implies that $e=2$ and $d=0$ and the parameters $(25,6,2,0)$ are not in Brouwer's list of parameters of strongly regular graphs \cite{BWP}. 
Alternatively, by ($b$) of Theorem 3.2 in \cite{PV3}, $\G(4,25)$ is not an srg since it has 5 different eigenvalues. 
The graph $\G(8,25)$ is not strongly regular since it is directed.
\end{exam}

\begin{exam}[$p=7$]  \label{exam: GP7s}
Here we give the GP-graphs $\G(k,7^m)$ with $k\mid 7^m-1$ and $m=1,2$.
For $m=1$, we have the graphs: 
	$$\G(1,7)=K_7, \qquad \G(2,7)=\vec{\mathcal{P}}_7, \qquad \G(3,7) = C_7 \qquad \text{ and } \qquad \G(6,7) = \vec{C}_7.$$ 
For $m=2$, the divisors of $48$ are $1,2,3,4,6,8,12,16,24,48$ and hence we have the ten 
graphs:
\begin{align*}
\G(1,49)  & = K_{49}, \\ 
\G(2,49)  & = \mathcal{P}_{49}, \\ 
\G(3,49)  & = \text{undirected, connected, $16$-regular (not srg)}, \\ 
\G(4,49)  & = srg(49,12,5,2) = L_{7,7}, \\ 
\G(6,49)  & = \text{undirected, connected, $8$-regular}, \\
\G(8,49)  & = K_7 \sqcup K_7 \sqcup K_7 \sqcup K_7 \sqcup K_7 \sqcup K_7 \sqcup K_7, \\
\G(12,49) & = \text{undirected, connected, $4$-regular},  \\ 
\G(16,49) & = \vec{\mathcal{P}}_7 \sqcup \vec{\mathcal{P}}_7 \sqcup \vec{\mathcal{P}}_7 \sqcup \vec{\mathcal{P}}_7 \sqcup \vec{\mathcal{P}}_7 \sqcup \vec{\mathcal{P}}_7 \sqcup \vec{\mathcal{P}}_7, \\ 
\G(24,49) & = C_7 \sqcup C_7 \sqcup C_7 \sqcup C_7 \sqcup C_7 \sqcup C_7 \sqcup C_7, \\
\G(48,49) & = \vec{C}_7 \sqcup \vec{C}_7 \sqcup \vec{C}_7 \sqcup \vec{C}_7 \sqcup \vec{C}_7 \sqcup \vec{C}_7 \sqcup \vec{C}_7.
\end{align*}
We leave to the reader the details of checking directedness and connectedness. 
Note that $\G(8,49)=\G(7^1+1,7^2)$ is as in \eqref{eq: ptKpt} and hence equal to 7 copies of $K_7$.
For the decomposition of $\G(16,49)$ we use Theorem \ref{thm: components}.

Although there is an srg with parameters $srg(49,16,3,6)$, by ($b$) of Theorem 3.1 in \cite{PV3} the graph $\G(3,7^2)$ has 4 different eigenvalues and hence it is not strongly regular. 
Note that since 
	$$\G(4,49)=\G(\tfrac{7^1+1}2,7^2),$$ 
by \eqref{eq: rook} it is the rook's graph $L_{7,7}$. 
On the other hand, by looking at Brouwer's tables of parameters \cite{BWP}, we see that $\G(6,49)$ and $\G(12,49)$ are not strongly regular graphs.
\hfill $\lozenge$
\end{exam}

\section{Non-bipartiteness of $\G(k,q)$} \label{sec:5}
Here we show that the graphs $\G(k,q)$ are non-bipartite, with the only exception of the graphs $\G(2^m-1,2^m)$ with $m \in \N$. 

A simple graph $\G=(V,E)$ is \textit{bipartite} if there exists a bipartition of $V$ into two independent sets, that is there exist  disjoint subsets $V_1,V_2$ of $V$ such that $V=V_1\cup V_2$ and if $a,b\in V_i$ with $i\in \{1,2\}$ then $ab \not \in E$
($\{V_1, V_2\}$ is called the bipartition of $V$).

The \textit{spectral radius} of $\G$ is defined as
\begin{equation} \label{eq: spectral radius}
\rho(\G)=\max \{ |\lambda|: \lambda \in Spec(\G)\}.
\end{equation}	 
An undirected graph $\G$ is bipartite if some (and hence all) of the following equivalent conditions are satisfied:
\vspace{1.5mm}
\begin{equation} \label{eq: C1C4}
\begin{aligned}
& (\text{B1}). \quad \text{The number $-\rho(\G)\in Spec(\G)$}. \\[1mm]
& (\text{B2}). \quad \text{The spectrum of $\G$ is symmetric (i.e.\@ $\lambda \in Spec(\G) \Leftrightarrow -\lambda \in Spec(\G)$)}. \\[1mm]
& (\text{B3}). \quad \text{ $\G$ has no cycles of odd length}. \\[1mm]
& (\text{B4}). \quad \text{The chromatic number of $\G$ is 2, that is $\chi(\G)=2$}.  
\end{aligned}
\end{equation}
Also, it is well-known that $\rho(\G) =\Delta(\G)$ with 
$$\Delta(\G) = \max \{\deg(v): v\in V(\G)\},$$
where $\deg(v)$ is the degree of $v$.
When $\G$ is $m$-regular, $\rho(\G)=m$ and hence (B1) 
can be changed by ``$-m \in Spec(\G)$''. 

The notion on bipartiteness can be extended to directed graphs as follows. A directed graph $D$ is called \textit{bipartite} if there exists a bipartition $V_1, V_2$ of $V(D)$ such that every arc of $D$ connects a vertex of $V_1$ to a vertex in $V_2$ or vice versa. In this case, it can be shown that the condition on $D$ to be bipartite is equivalent to  (B1) and (B2) of the above list (see Theorem~3.1 in \cite{Br}). Moreover the condition (B3)  can be changed by 
\begin{enumerate}[(B3').]
	\item $D$ has only even directed cycles.
\end{enumerate}

In Theorem 5.8 of \cite{PV3}, we proved that $\G(k,q)$ and $\bar \G(k,q)$ are non-bipartite for any pair of semiprimitive integers $(k,q)$. We now extend this result by showing that all connected generalized Paley graphs $\G(k,q)$ are non-bipartite, except for the trivial case $\G(1,2)=K_2$. 
\begin{prop} \label{prop caract. GP bip}
 Let $q=p^m$ with $p$ prime, $m\in \N$, and let $k \in \N$ be such that $k\mid q-1$. 
 If $\G(k,q)$ is connected then $\G(k,q)$ is non-bipartite with the only exception of $\G(1,2)=K_2$.
\end{prop}

\begin{proof}
We will prove that $\G(k,q)$ is bipartite if and only if $q=2$ and $k=1$.
In one direction it is obvious: if $k=1$ and $q=2$, we have $\G(1,2)=K_2=C_2$, which is trivially bipartite.

\sk 

Conversely, suppose that $\G(k,q)$ is bipartite.
Notice that if $q$ is odd, then by successively adding $1$ to a vertex we can obtain a (directed) cycle of odd length $p$ in $\G(k,q)$, and so $\G(k,q)$ is not bipartite, a contradiction to (B3) or (B3'). 
In fact, if $x\in \ff_q$, then 
$$x, \: x+1, \:  x+2, \:  \ldots, \:  x+p-1, \:  x+p=x$$ 
is an oriented $p$-cycle in $\G(k,q)$, say $\vec{C}_p(x)$. 
Clearly, all the vertices are different, two consecutive vertices $x+j$ and $x+j+1$ form an edge since 
	$$ (x+j+1) - (x+j)=1 \in R_k, $$ 
and $x+p=x$ since we are in characteristic $p$.
If $\G(k,q)$ is undirected, we have the reverse cycle $\overset{_{\leftarrow}}{C}_p(x)$, and hence an undirected $p$-cycle  $C_p(x) = \overset{_{\rightarrow}}{C}_p(x) \cup \overset{_{\leftarrow}}{C}_p(x)$.
	
\sk 
So, from now on we assume that $q=2^m$ with $m \ge 1$, hence $\G(k,q)$ is undirected.
Recall that $\G(k,q)$ is $n$-regular with $n=\frac{q-1}k$ and that $\frac{q-1}k$ is the principal eigenvalue $\lambda_1$ of $\G(k,q)$. Hence, $\G(k,q)$ is bipartite if and only if has no odd cycles,  by (B3).

Notice that if $k=2^{m}-1$, then the graph $\G(2^{m}-1,2^{m})$ is not connected for $m>1$ by Theorem \ref{thm: components}, so we can assume that $k<2^{m}-1$ when $m>1$.
	
\sk	
\noindent
\textit{Claim:} If $k<2^{m}-1$ with $m>1$ then the undirected graph $\G(k,2^m)$ contains a cycle of odd length $r$ 
(in fact, $r$ is prime with $r \mid \frac{2^m-1}{k}$).

Assume that $k<2^{m}-1$ with $m>1$.
The elements of $R_k$ are exactly the roots of the polynomial $p(x)=x^{\frac{2^m-1}{k}}-1\in \ff_{2}[x]$, which in $\ff_{2^m}$ factorizes as
	$$ p(x) = \prod_{a\in R_k}(x-a).$$ 
In particular, the sum of all the elements in $R_k$ coincides with the second leading coefficient of $p(x)$ and so
\begin{equation*} \label{eq: Rk}
	\sum_{a\in R_k}a=0.
\end{equation*}

Since $\frac{2^m-1}{k}>1$ and it is an odd number, then there exists an odd prime $r$ with $r \mid \frac{2^m-1}{k}$, i.e.\@ $\frac{2^m-1}{k} = rt$ for some $t \in \mathbb{N}$, 
and hence we have 
	$$R_{kt} \subseteq R_k.$$ 
As before, the elements of $R_{kt}$ are exactly the roots of the polynomial $q(x)=x^{\frac{2^m-1}{kt}}-1$, 
that is
$ q(x) = \prod_{a\in R_{kt}}(x-a) \in \ff_{2^m}[x]$,  
and hence there exists $r=|R_{kt}|$ elements in $R_k$ such that 
\begin{equation} \label{eq: Rkt}
	\sum_{a\in R_{kt}}a=0.
\end{equation}
Let $\alpha$ be a generator of $R_{\frac{2^m-1}{r}}$, 
thus $\alpha = \omega^{\frac{2^m-1}{r}}$ is an element of order $r$ in $\ff_{2^m}$, 
where $\omega$ is a primitive element of $\ff_{2^m}$.

Let us see that $0, 1, 1+\alpha, 1+\alpha+\alpha^{2}, \ldots, 1+\alpha+\cdots+\alpha^{r-2}$ and $1+\alpha+\cdots+\alpha^{r-1}$ form a cycle in $\G(k,2^m)$. 
Indeed, since $\alpha\in R_{k}$, any power $\alpha^i$ is also in $R_k$. So, the powers of $\alpha$ induce edges in the graph $\G(k,2^m)$; namely, there is an edge between any vertex $x$ and $y=x+\alpha^i$ for any $i$.
Thus, we have the following walk in $\G(k,2^m)$.
\begin{equation} \label{eq: cycle}
	0, \: 1, \: 1+\alpha, \: 1+\alpha+\alpha^{2}, \: \ldots, \: 1+\alpha+\cdots+\alpha^{r-2}, \: 1+\alpha+\cdots+\alpha^{r-1}
\end{equation}	
and, since $1+\alpha+\cdots+\alpha^{r-1}=0$ by \eqref{eq: Rkt}, it is a closed walk in $\G(k,2^m)$ of length $r$. 

It is enough to see that all of the intermediate vertices $1+\alpha+\cdots +\alpha^i$ with $i\ne r-1$ are all different.
Suppose, by contradiction, that there are $i,j \in \{1,\ldots,r-1\}$ with $i<j$ such that
	$$ 1+\alpha+\cdots +\alpha^i = 1+\alpha+\cdots +\alpha^i + \alpha^{i+1} + \cdots +\alpha^j. $$
Then, we have that $\alpha^{i+1} + \cdots +\alpha^j=0$, that is $\alpha^{i+1} (1+\alpha + \cdots +\alpha^{j-i-1})=0$, 
which implies that
	$${\sum_{\ell=0}^{j-i-1} \alpha^{\ell}=0}.$$ 
Thus, $\alpha$ is a zero of the polynomial $s(x)=x^{j-i}-1$ and hence $\alpha^{j-i}=1$ with $j-i<r$, which is absurd. 
In this way, we see that all of the intermediate vertices in \eqref{eq: cycle} are all different, and hence we obtain an oriented 

By the Claim, if $m>1$ and $k<2^{m}-1$, then the graph $\G(k,2^m)$ has some odd cycle, which implies that $\G(k,2^{m})$ is non-bipartite. This completes the proof.
\end{proof}

By putting together the previous proposition and Theorem \ref{thm: components} we obtain the following characterization of non-bipartiteness of GP-graphs.
\begin{thm} \label{teo: GPbipartito}
	Let $q=p^m$ be a fixed prime power with $m \in \N$ and $k \in \N$ with $k\mid q-1$. 
	Then, $\G(k,q)$ is non-bipartite, except for $\G(2^m-1,2^m) \simeq K_2 \sqcup \cdots \sqcup K_2$ $(2^{m-1}$ copies$)$.
\end{thm}

\begin{proof}
If $\frac{q-1}k \dagger q-1$ then $\G(k,q)$ is connected by \eqref{eq: C1C2} (see \eqref{eq: prim div}--\eqref{eq: ord p}). By Proposition~\ref{prop caract. GP bip} the GP-graph $\G(k,q)$ is non-bipartite except for $\G(1,2)=K_2$. 
	
Let $n=\frac{q-1}k$. If $n$ is not a primitive divisor of $q-1$ then $\G(k,q)$ is not connected. Also, if $a$ is the minimal positive integer such that $n \mid p^a-1$, i.e.\@ $a=ord_n(p)$, then by Theorem~\ref{thm: components} we know that $\G(k,q)$ has $p^{m-a}$ connected components, all isomorphic to $\G(k_a,p^a)$ with $k_a = \frac{p^{a}-1}n$. 
	
Assume that $\G(k,q)$ is bipartite.
A disjoint union of graphs is bipartite if and only if each component is bipartite. 
Thus, $\G(k_a,p^a)$ is connected and bipartite and hence by Proposition~\ref{prop caract. GP bip} we must have that 					
	$$\G(k_a,p^a)=\G(1,2)=K_2.$$ 
Hence, $p=2$, $a=1$ and $k_a=1$. 
In this way $q=2^m$ and, since $k_a=\frac 1n$, we have that $1=n=\frac{2^m-1}k$ from where we deduce that $k=2^m-1$,
and the result follows. 
\end{proof}

Thus, by \eqref{eq: C1C4}, any GP-graph $\G=\G(k,q)$, which is not of the form $\G_m=\G(2^m-1,2^m)$ for some $m \in\N$, contains some odd cycle, has asymmetric spectrum and its chromatic number satisfies $\chi(\G)\ge 3$. 

As a direct consequence of the previous theorem we get that, generically, the bipartite double of a connected GP-graph is connected.
\begin{coro} \label{coro: double}
Let $\G(k,q)$ be a GP-graph such that $(k,q)\ne (2^m-1,2^m)$ for $m \in \N$. If $\frac{q-1}k \dagger q-1$ then the bipartite double $B(\G(k,q))=\G(k,q) \times K_2$ is connected.
\end{coro}

\begin{proof}
By Theorem 3.4 in \cite{BHM}, the bipartite double $B(X)=X \times K_2$ of a graph $X$ is connected if and only if $X$ is connected and non-bipartite. 
We know that $\G(k,q)$ is connected if and only if $\frac{q-1}k$ is a primitive divisor of $q-1$. 
By Theorem \ref{teo: GPbipartito}, $\G(k,q)$ is non-bipartite for any $(k,q)\ne (2^m-1,2^m)$ with $m\in \N$. Hence, $\G(k,q) \times K_2$ is connected for $(k,q)$ under the hypothesis.
\end{proof}

This completely characterizes the connected bipartite doubles of GP-graphs. 

\begin{exam} \label{exam: semip}
If $(k,q)$ is a semiprimitive pair, i.e.\@ $-1$ is a power of $p$ modulo $k$, then $\G(k,q)$ is a strongly regular graph with 3 different eigenvalues (see Definition 5.1 and Theorems~5.4 and 5.8 in \cite{PV3}). Hence, $\G(k,q)$ is connected and therefore the double $B(\G(k,q))$ is connected. 
For instance $(p^t+1,p^m)$ with $t\mid m$ and $\frac mt$ even (and $t\ne \frac m2$) is semiprimitive. In particular, for $t=1$ and $p=2$ with $m$ even, $(3,2^m)$ is a semiprimitive pair and 
	$B(\G(3,2^{m})) = \G(3,2^{m}) \times K_2$ 
is connected for every $m$ even.  
\hfill $\lozenge$
\end{exam}

To conclude the paper we look at the first binary GP-graphs, that is $\G(k,2^m)$ with $k\mid 2^m-1$ for the smallest values of $m$.

\begin{exam}[$p=2$, $1\le m\le 8$] \label{exam: GP2s}
We recall that binary GP-graphs are always undirected (see \eqref{eq: C1C2}). We give the graphs 
and leave some details to the reader. In all the cases, the decomposition $\G(2^m-1,2^{m})=2^{m-1}K_2$ follows from Theorem \ref{teo: GPbipartito}.

We have $\G(1,2)=K_2$ for $m=1$; $\G(1,4)=K_4$ and $\G(3,4) \simeq K_2 \sqcup K_2$ for $m=2$; and 
$\G(1,8)=K_8$ and $\G(7,8) \simeq K_2 \sqcup K_2 \sqcup K_2  \sqcup K_2$ for $m=3$.

For $m=4$, we have four GP-graphs: 
\begin{align*}
	\G(1,16)  & = K_{16}, \\ 
	\G(3,16)  & = srg(16,5,0,2) = \text{Clebsch graph}, \\ 
	\G(5,16)  & \simeq K_4 \sqcup K_4 \sqcup K_4 \sqcup K_4, \\ 
	\G(15,16) & = K_2 \sqcup K_2 \sqcup K_2 \sqcup K_2 \sqcup K_2 \sqcup K_2 \sqcup K_2 \sqcup K_2. 
\end{align*}
Note that $\G(3,16)=\G(2^1+1,2^4)$ 
belongs to the family defined in \eqref{eq: GPqml} of Example~\ref{ex: qml}, and hence it is a strongly regular graph. It is the Clebsch graph, having girth $g=4$ and chromatic number $\chi=4$. It is also Ramanujan (see Section 8 in \cite{PV4b}). 
There are 2{.}585{.}136{.}675 connected 5-regular graphs of 16 vertices (see Markus Meringer's web page \cite{MMWP}). 
Notably, out of them, there are only 4 which are strongly regular, the Clebsch graph, the Shrikhande graph with parameters $srg(16,6,2,2)$, and their complements. 

Consider 
	$\ff_{16}=\ff_{2}[x]/(x^{4}+x+1)$
and let $\alpha$ be a root of the irreducible polynomial $x^4+x+1$, that is $\alpha^4=\alpha+1$. The cubes in $\ff_{16}$ are 
	$$R_3 = \{1,\alpha^{3}, \alpha^{3}+\alpha,\alpha^{3}+\alpha^2, \alpha^{3}+\alpha^{2}+\alpha+1\}$$
and hence we can depict the Clebsch graph as the GP-graph $\G(3,16)$ as follows:

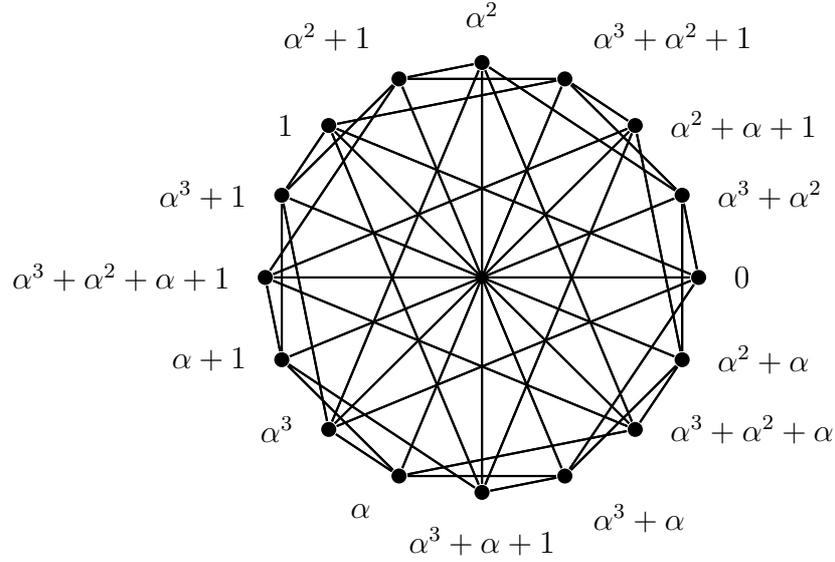
\begin{center}
\begin{figure}
\begin{tikzpicture}[scale=.875, auto, thick, 
	main node/.style={fill=black, circle, inner sep=2pt}, 
	label distance=2mm] 

	\foreach \i/\name/\pos in {
		0/$0$/right, 
		1/$\alpha^{3}+\alpha^{2}$/ right, 
		2/$\alpha^{2}+\alpha+1$/ right, 
		3/$\alpha^{3}+\alpha^{2}+1$/above right, 
		4/$\alpha^{2}$/above, 
		5/$\alpha^{2}+1$/above left, 
		6/$1$/ left, 
		7/$\alpha^{3}+1$/left, 
		8/$\alpha^{3}+\alpha^{2}+\alpha+1$/left,
		9/$\alpha+1$/left,
		10/$\alpha^{3}$/left,
		11/$\alpha$/below left,
		12/$\alpha^{3}+\alpha+1$/below ,
		13/$\alpha^{3}+\alpha$/below right,
		14/$\alpha^{3}+\alpha^{2}+\alpha$/right,
		15/$\alpha^{2}+\alpha$/right}
	{
		\node[main node] (n\i) at ({360/16 * \i}:3cm) [label=\pos:\name] {}; 
	}
	
	
	
	\path[thick, draw=black] (n0) edge (n1); 
	\path[thick, draw=black] (n0) edge (n6); 
	\path[thick, draw=black] (n0) edge (n8); 
	\path[thick, draw=black] (n0) edge (n10); 
	\path[thick, draw=black] (n0) edge (n13); 

	\path[thick, draw=black] (n1) edge (n0); 
	\path[thick, draw=black] (n1) edge (n9); 
	\path[thick, draw=black] (n1) edge (n4); 
	\path[thick, draw=black] (n1) edge (n3); 
	\path[thick, draw=black] (n1) edge (n15); 
	
	\path[thick, draw=black] (n2) edge (n12); 
	\path[thick, draw=black] (n2) edge (n10); 
	\path[thick, draw=black] (n2) edge (n3); 
	\path[thick, draw=black] (n2) edge (n8); 
	\path[thick, draw=black] (n2) edge (n15); 

	\path[thick, draw=black] (n3) edge (n5); 
	\path[thick, draw=black] (n3) edge (n6); 
	\path[thick, draw=black] (n3) edge (n2); 
	\path[thick, draw=black] (n3) edge (n11); 
	\path[thick, draw=black] (n3) edge (n1); 
	
	\path[thick, draw=black] (n4) edge (n1); 
	\path[thick, draw=black] (n4) edge (n5); 
	\path[thick, draw=black] (n4) edge (n14); 
	\path[thick, draw=black] (n4) edge (n12); 
	\path[thick, draw=black] (n4) edge (n10); 
	
	\path[thick, draw=black] (n5) edge (n4); 
	\path[thick, draw=black] (n5) edge (n3); 
	\path[thick, draw=black] (n5) edge (n7); 
	\path[thick, draw=black] (n5) edge (n8); 
	\path[thick, draw=black] (n5) edge (n13); 

	\path[thick, draw=black] (n6) edge (n0); 
	\path[thick, draw=black] (n6) edge (n3); 
	\path[thick, draw=black] (n6) edge (n7); 
	\path[thick, draw=black] (n6) edge (n12); 
	\path[thick, draw=black] (n6) edge (n14); 
	
	\path[thick, draw=black] (n7) edge (n6); 
	\path[thick, draw=black] (n7) edge (n5); 
	\path[thick, draw=black] (n7) edge (n10); 
	\path[thick, draw=black] (n7) edge (n9); 
	\path[thick, draw=black] (n7) edge (n15); 
	
	\path[thick, draw=black] (n8) edge (n0); 
	\path[thick, draw=black] (n8) edge (n2); 
	\path[thick, draw=black] (n8) edge (n9); 
	\path[thick, draw=black] (n8) edge (n5); 
	\path[thick, draw=black] (n8) edge (n14); 
	
	\path[thick, draw=black] (n9) edge (n7); 
	\path[thick, draw=black] (n9) edge (n8); 
	\path[thick, draw=black] (n9) edge (n1); 
	\path[thick, draw=black] (n9) edge (n11); 
	\path[thick, draw=black] (n9) edge (n12); 
	
	\path[thick, draw=black] (n10) edge (n7); 
	\path[thick, draw=black] (n10) edge (n0); 
	\path[thick, draw=black] (n10) edge (n2); 
	\path[thick, draw=black] (n10) edge (n4); 
	\path[thick, draw=black] (n10) edge (n11); 
	
	\path[thick, draw=black] (n11) edge (n10); 
	\path[thick, draw=black] (n11) edge (n9); 
	\path[thick, draw=black] (n11) edge (n3); 
	\path[thick, draw=black] (n11) edge (n13); 
	\path[thick, draw=black] (n11) edge (n14); 
	
	\path[thick, draw=black] (n12) edge (n9); 
	\path[thick, draw=black] (n12) edge (n6); 
	\path[thick, draw=black] (n12) edge (n4); 
	\path[thick, draw=black] (n12) edge (n2); 
	\path[thick, draw=black] (n12) edge (n13); 
	
	\path[thick, draw=black] (n13) edge (n11); 
	\path[thick, draw=black] (n13) edge (n12); 
	\path[thick, draw=black] (n13) edge (n5); 
	\path[thick, draw=black] (n13) edge (n15); 
	\path[thick, draw=black] (n13) edge (n0); 
	
	\path[thick, draw=black] (n14) edge (n11); 
	\path[thick, draw=black] (n14) edge (n8); 
	\path[thick, draw=black] (n14) edge (n4); 
	\path[thick, draw=black] (n14) edge (n15); 
	\path[thick, draw=black] (n14) edge (n6); 
	
	\path[thick, draw=black] (n15) edge (n14); 
	\path[thick, draw=black] (n15) edge (n1); 
	\path[thick, draw=black] (n15) edge (n2); 
	\path[thick, draw=black] (n15) edge (n7); 
	\path[thick, draw=black] (n15) edge (n13); 
\end{tikzpicture}
\caption{The Clebsch graph $\G(3,16)$.}
\end{figure}
\end{center}

For $m=5$, there are only two binary GP-graphs: $\G(1,32)=K_{32}$ and $\G(31,32) \simeq 2^4K_2.$ 
The case $m=6$ is more interesting. The divisors of $2^6-1=63$ are $k=1,3,7,9,21,63$, and thus we have the six 
GP-graphs:
\begin{align*}
 \G(1,64) & = K_{64}, \\ 
 \G(3,64) & = \text{connected $21$-regular} = 
 srg(64,21,8,6), \\ 
 \G(7,64) & = \text{connected $9$-regular (not srg)}, \\ 
 \G(9,64) & = K_8 \sqcup \cdots \sqcup K_8 \quad (\text{$2^3$-times}), \\ 
 \G(21,64)& = K_4 \sqcup \cdots \sqcup K_4 \quad (\text{$2^4$-times}), \\ 
 \G(63,64)& = K_2 \sqcup \cdots \sqcup K_2 \quad (\text{$2^5$-times}). 
\end{align*}
Note that $\G(3,64)=\G(2^1+1,2^6)$ belongs to the family defined in \eqref{eq: GPqml} of Example \ref{ex: qml}, 
hence is a strongly regular graph. It is also Ramanujan (see Section 8 in \cite{PV4b}). On the other hand, 
	$\G(9,64)=\G(2^3+1,2^6)$
is as in \eqref{eq: q^mKq^m} of Remark \ref{rem: 1}. For $\G(21,64)$ we use Theorem~\ref{thm: components}.
Finally, $ \G(7,64)$ is not a strongly regular graph, since $srg(64,9,e,d)$ and \eqref{eq: srg cond} implies $e=2$ and $d=0$ and from the Brouwer's tables \cite{BWP} we see that $srg(64,9,2,0)$ do not exist.

For $m=7$, there are only two binary generalized Paley graphs: $\G(1,128)=K_{128}$ and $\G(127,128) \simeq 2^6 K_2$. 

Finally, for $m=8$, we have $2^8-1=255=3\cdot 5 \cdot 17$ and thus we have the eight
GP-graphs:
\begin{align*}
\G(1,256)  & = K_{256}, \\ 
\G(3,256)  & = \text{connected $85$-regular} = srg(256, 85, 24, 30), \\ 
\G(5,256)  & = \text{connected $51$-regular} = srg(256, 51, 2, 12), \\ 
\G(15,256) & = \text{connected $17$-regular (not srg)}, \\ 
\G(17,256) & = K_{16} \sqcup \cdots \sqcup K_{16} \quad (\text{$2^4$-times}), \\ 
\G(51,256) & = \G(3,16) \sqcup \cdots \sqcup \G(3,16) \quad (\text{ $2^4$-times}). \\
\G(85,256) & = K_4 \sqcup \cdots \sqcup K_4 \quad (\text{$2^6$-times}), \\ 
\G(255,256)& = K_2 \sqcup \cdots \sqcup K_2 \quad (\text{$2^7$-times}). 
\end{align*} 
Note that the graphs 
	$$\G(3,256)=\G(2^1+1,2^8) \qquad \text{and} \qquad \G(5,256)=\G(2^2+1,2^8)$$ 
belong to the family defined in \eqref{eq: GPqml} of Example \ref{ex: qml}, and hence they are strongly regular graphs. 
These graphs are also Ramanujan (see Section 8 in \cite{PV4b}).
On the other hand, 
	$$\G(17,256)=\G(2^4+1,2^8)$$ 
is as in \eqref{eq: q^mKq^m} of Remark \ref{rem: 1}.
For $\G(51,256)$ and $\G(85,256)$ we use Theorem~\ref{thm: components}.

Again, if $\G(15,256)$ were strongly regular with parameters $srg(256,17,e,d)$, condition \eqref{eq: srg cond} implies that  
$e=2$ and $d=0$, and there are no strongly regular graphs with these parameters (\cite{BWP}).

For the graph $\G(51,256)$ we use Theorem \ref{thm: components}. Notably, the graph $\G(51,256)$ is the first disconnected binary or ternary GP-graph which is not the union of cliques, Paley graphs or cycles, but the union of isomorphic copies of the Clebsch graph.
\hfill $\lozenge$
\end{exam}

From the previous examples, by Corollary \ref{coro: double}, we have that 
	$B(\G(3,16))$, $B(\G(3,64))$, $B(\G(7,64))$, $B(\G(3,256))$, $B(\G(3,256))$ and $B(\G(15,256))$
are all connected.

\begin{rem}
A Hamming graph $H(b,q)$ is a graph with vertex set $V=K^b$ where $K$ is any set of size $q$
and where two $b$-tuples form an edge if and only if they differ in exactly one coordinate.
Notice that $H(b,q)=\square^b K_{q}$ (Cartesian product) and hence, Hamming graphs have integral spectra. 
Those connected GP-graphs which are Hamming were classified by Lim and Praeger in \cite[Theorem 1.2]{LP}. In this case $k=\tfrac{q^b-1}{b(q-1)}$ for integers $b \mid \tfrac{q^{b}-1}{q-1}$, $q=p^{m}$ with $p$ prime, and 
\begin{equation} \label{eq: HammingGP}
	\G(\tfrac{q^{b}-1}{b(q-1)}, q^{b}) = H(b,q).
\end{equation}	
Clearly, $H(1,q)=K_q$ and $H(2,q)=L_{q,q}$.
Looking at the graphs from Examples~\ref{exam: GP3s}--\ref{exam: GP7s} and \ref{exam: GP2s} we see that the only non-trivial Hamming GP-graph (i.e.\@ those with $b\ge 3$) is 
	$$\G(7,64) = H(3,4) = K_4 \square K_4 \square K_4$$ 
(taking $p=2$, $m=2$ and $b=3$). Interestingly, note that taking $p=3$, $m=1$ and $b=4$ in \eqref{eq: HammingGP} 
we see that $\G(10,81)$ has the parameters of a Hamming graph, namely $H(4,3)$. 
However, $\G(10,81) \ne H(4,3)$ since $\G(10,81)$ is not connected.
\end{rem}

\section*{Final remarks}
In this paper we have used the connected component decomposition of GP-graphs $\G(k,q)$ over finite fields $\ff_q$ to characterize all bipartite GP-graphs.
This decomposition would be useful to study connectedness  and non-bipartiteness of other structural properties in certain products of GP-graphs over rings, namely 
	$$G_R(k)=Cay(R,\{x^k:x\in R^*\})$$
where $R$ is a finite commutative ring with identity.

In Examples \ref{exam: GP3s}--\ref{exam: GP7s} and \ref{exam: GP2s} we have studied the GP-graphs $\G(k,p^m)$ for $p=2,3,5,7$ and small values of $m$. In almost all cases we have identified the graph (if it is connected) or given the decomposition into connected components as a union of simpler known graphs. This allows one to compute invariants such as diameter (and hence Waring numbers), girth, chromatic number and the spectrum. The spectrum is well-known for cliques, Paley graphs (directed or not), rook's graph, cycles (directed or not), strongly regular graphs and Hamming graphs. Also, the spectrum of $\G(3,q)$ and $\G(4,q)$ was recently given in \cite{PV3}. 
This only leaves few graphs in the examples of this work for which we cannot compute the invariants or the spectrum with the actual techniques. For instance, we cannot compute the spectrum for only six graphs in the examples: 
$\G(8,81)$, $\G(16,81)$, $\G(8,25)$, $\G(6,49)$, $\G(12,49)$ 
and $\G(15,64)$, since they do not fall in the families with known spectrum.

Also, in a future work, we will use this decomposition to characterize GP-graphs $\G(k,p^m)$ which are built of copies of smaller GP-graphs $\G(k_a,p^a)$, i.e.\@ with $k_a \mid k$ and $a<m$, belonging to particular families such as those with $k$ small, or semiprimitive GP-graphs, Hamming GP-graphs, Cartesian decomposable GP-graphs, etc. This will allow us to give the spectrum in all the cases and the automorphism groups and some invariants in many cases.

\end{document}